\documentclass[11pt]{amsart}

\usepackage{verbatim}
\usepackage{url}
\usepackage{hyperref}

\usepackage[autostyle]{csquotes}
\usepackage[
    backend=biber,
    url=true,
    style=alphabetic,
    sorting=nty,
    uniquename=init]{biblatex}
\bibliography{bibliography}

\usepackage{mathrsfs}
\usepackage{amsmath, amssymb, amsgen, amsxtra, amsfonts, amsbsy}
\usepackage{amsthm}
\usepackage{mathtools}
\usepackage{changepage}
\usepackage{enumitem}

\makeatletter
\newcommand{\customanchor}[1]{\vspace{-20pt}\Hy@raisedlink{\hypertarget{#1}{}}}
\makeatother

\theoremstyle{definition}
\newtheorem{df}{Definition}[section]

\newtheorem*{ack}{Acknowledgements}
\newtheorem*{org}{Organisation}

\theoremstyle{plain}
\newtheorem{thm}[df]{Theorem}
\newtheorem{lem}[df]{Lemma}
\newtheorem{pro}[df]{Proposition}
\newtheorem{cor}[df]{Corollary}

\theoremstyle{remark}
\newtheorem{rem}[df]{Remark}

\newtheorem*{claim*}{Claim}

\usepackage{nicematrix}
\NiceMatrixOptions{cell-space-limits = 1pt}
\usepackage{arydshln} 
\usepackage{tabularray}
\UseTblrLibrary{booktabs}

\usepackage[width=1.03\textwidth, font=small]{caption}

\usepackage{tikz-cd}
\usepackage{array, multirow, graphicx}

\begin{document}
\makeatletter
\@namedef{subjectclassname@2020}{\textup{2020} Mathematics Subject Classification}
\makeatother

\title[Moduli of Supersingular Abelian Varieties]{A Computational Framework for the Moduli of Supersingular Abelian Varieties}
\author{Michael Burger}
\address{TU M\"unchen, School of Computation, Information and Technology (CIT), Boltzmannstr. 3, 85748 Garching bei M\"unchen, Germany}
\email{michael.burger@tum.de}

\subjclass[2020]{14K10, 11G10, 14L05}
\keywords{moduli space, supersingular abelian variety, superspecial abelian variety, Dieudonné modules}

\begin{abstract}
We establish structure theorems for polarised flag type quotients arising in the study of the moduli space of supersingular abelian varieties. In particular, we prove a refined normal form for the top level of these flags and derive an explicit, computable descent criterion for quasi-polarisations. These results provide a complete classification of the first and last step of these flags. 

As an application, we compute polarised flag type quotients in dimensions $g\leq 5$, describing the fibers of the truncation morphisms as classification objects of supersingular abelian varieties. 
\end{abstract}

\setcounter{tocdepth}{1}
\vspace*{-35pt}
\maketitle
\vspace{-30pt}
\tableofcontents
\vspace{-30pt}

\section{Introduction}

The moduli space $\mathcal{A}_g$ of principally polarised abelian varieties and its supersingular locus $\mathcal{S}_g$ are central objects in arithmetic geometry. While the existence of $\mathcal{S}_g$ as a (coarse) moduli space follows from geometric invariant theory, this approach does not provide an explicit description of its geometry. After preliminary work due to Oda and Oort in \parencite{OdaOort}, a powerful strategy aimed at solving this problem was introduced by Li and Oort in \parencite{LiOort} via the framework of \emph{polarised flag type quotients} (PFTQs for short). More recently, Viehmann used Rapoport-Zink uniformisation to analyse the generic $a=1$ locus inside $\mathcal{S}_g$ in \parencite{Viehmann}. While this approach gives an explicit description of a dense open subset of $\mathcal{S}_g$, PFTQs provide a framework for studying all points of $\mathcal{S}_g$, including those outside the dense open locus, and thus are particularly well suited for describing the moduli space in its entirety. 
However, while the theoretical framework is well-established, explicit calculations of these PFTQs remain cumbersome. Therefore, the existing literature only contains partial results (see, for example \parencite{Harashita, Dina, Pieper}) focusing on isolated steps of the PFTQs in selected dimensions rather than systematically treating these structures. 
This article aims to close this gap by developing a computational framework for PFTQs in arbitrary dimension, providing a unified approach to the calculations that arise throughout the construction of these flags.

\subsection{Structural results}
The main contribution of this article is to bundle the computations into two structural results that can be used as black boxes.
\begin{itemize}
    \item[(1)] Theorem \ref{thm:Structure} and its corollary establish a normal form for the quasi-polarised superspecial Dieudonné module at the top of a PFTQ, refining the structure theorem due to Li and Oort \parencite[Theorem 6.1]{LiOort} by exploiting the additional constraints of the PFTQ structure.
    \item[(2)] Lemma \ref{lem:descentcriterion} replaces the abstract descent condition in the definition of PFTQs by an explicit integrality condition on the quasi-polarisation, giving a computable criterion for descending down the PFTQ. 
\end{itemize}
These two results establish a framework for computing the moduli of supersingular abelian varieties in arbitrary dimension by making each step of the PFTQ explicit and calculable. In particular, they already provide a complete classification of the first and last step of these flags in all dimensions.

\subsection{Application in low dimensions} To show the strength of the established framework, we carry out the explicit calculations in dimensions $g \leq 5$ and systematically discuss each dimension: The lowest-dimensional case is not of particular interest as supersingular elliptic curves can be fully described by classical methods, see \parencite{Deuring, Igusa}. Thus, our calculations in Section \ref{sec:calculations} begin with dimension $g=2$. For $g = 2,3$, the structure is fully determined by the black-box theorems and we recover known descriptions of the moduli spaces, see \parencite{MoretBailly, IbukiyamaKaremaker}. Dimension $g = 4$ is particularly interesting as we have to introduce a new approach to describe all PFTQs and not only the ones leading to supergeneral abelian varieties, which involves splitting parameter choices into generic and arbitrary ones by their rationality over $\mathbb{F}_{p^2}$. Lastly, we compute PFTQs in dimension $g=5$, obtaining novel equations describing both generic and arbitrary parameter choices. 

\subsection{Fibers of the truncation morphisms}
These calculations yield fibers of truncation morphisms which are higher-dimensional generalisations of well-known classification objects like the Moret-Bailly parameter.

\begin{thm}
Let $k$ be algebraically closed of characteristic $p$. Then, Table \ref{tab:gen} collects the fibers of the truncation morphisms for generic parameters. 
\begin{table}[h!]
    \vspace{-8pt}
    \normalfont
    \scriptsize
    \centering
    \begin{tblr}{
    colspec = {ccc},
    colsep = 3pt,
    rowsep = 1.5pt,
    vlines
    }
    \hline
    \textbf{$g$} & \textbf{morph.} & \textbf{fiber description} \\
    \hline \hline
    \hyperref[sec:g2]{$2$} & $\mathcal{V}_0\to \mathcal{V}_1$ & $\mathbb{P}^1$ \\
    \hline
    \SetCell[r=2]{c} \hyperref[sec:g3]{$3$}
        & $\mathcal{V}_1 \to \mathcal{V}_2$ & $\alpha_1^{p+1} + \alpha_2^{p+1} + \alpha_3^{p+1} $ \\
        & $\mathcal{V}_0 \to \mathcal{V}_1$ & $\mathbb{P}^1$ \\
    \hline
    \SetCell[r=3]{c} \hyperref[sec:g4]{$4$}
        & $\mathcal{V}_2\to \mathcal{V}_3$ & $\alpha_1 \alpha_2^{p^2} - \alpha_2\alpha_1^{p^2} + \alpha_3\alpha_4^{p^2} - \alpha_4\alpha_3^{p^2}$ \\
        & $\mathcal{V}_1 \to \mathcal{V}_2$ & $\alpha_1\alpha_6^p + \alpha_3\alpha_7^p - \alpha_6 \alpha_5^{p-1}\alpha_1^p - \alpha_7 \alpha_5^{p-1}\alpha_3^p $ \\
        & $\mathcal{V}_0 \to \mathcal{V}_1$ & $\mathbb{P}^1$ \\
    \hline
   \SetCell[r=5]{c} \hyperref[sec:g5]{$5$}
       & \SetCell[r=2]{c} $\mathcal{V}_3 \to \mathcal{V}_4$
           & $\alpha_1^{p^3+1} + \alpha_2^{p^3+1} + \alpha_3^{p^3+1} + \alpha_4^{p^3+1} + \alpha_5^{p^3+1}$ \\
           & & $\alpha_1^{p+1} + \alpha_2^{p+1} + \alpha_3^{p+1} + \alpha_4^{p+1} + \alpha_5^{p+1}$ \\
       & $\mathcal{V}_2 \to \mathcal{V}_3$ & $-\alpha_3\alpha_7^{p^2}-\alpha_4\alpha_8^{p^2}-\alpha_5\alpha_9^{p^2} + \alpha_7\alpha_6^{p^2-1}\alpha_3^{p^2} + \alpha_8\alpha_6^{p^2-1}\alpha_4^{p^2}+\alpha_9\alpha_6^{p^2-1}\alpha_5^{p^2}$ \\
       & $\mathcal{V}_1 \to \mathcal{V}_2$ & $\alpha_6\alpha_4\alpha_{11}^p \mathord{+}\alpha_{11}\alpha_{10}^{p-1}\alpha_6^p\alpha_4^p \mathord{+} \alpha_6\alpha_5\alpha_{12}^p \mathord{+} \alpha_{12}\alpha_{10}^{p-1}\alpha_6^p\alpha_5^p \mathord{+} \alpha_{10}^p\alpha_7^{p+1} \mathord{+} \alpha_{10}^p\alpha_8^{p+1}\mathord{+}\alpha_{10}^p\alpha_9^{p+1}$ \\
       & $\mathcal{V}_0 \to \mathcal{V}_1$ & $\mathbb{P}^1$ \\
    \hline        
    \end{tblr}
    \vspace{-12pt}
    \caption{Fibers of the truncation morphisms for generic parameters.}
    \label{tab:gen}
    \vspace{-32pt}
\end{table}
\end{thm}
\pagebreak

\noindent Another key result is the description of the fibers of the truncation morphisms for arbitrary parameters. While the equations computed with generic parameter choices only capture the behaviour of supergeneral abelian varieties, the following equations describe all supersingular abelian varieties. 

\begin{thm}
Let $k$ be algebraically closed of characteristic $p$. Then, Table \ref{tab:arb} collects the fibers of the truncation morphisms for arbitrary parameters.
\begin{table}[h!]
    \centering
    \normalfont
    \scriptsize
    \vspace{-2pt}
    \makebox[\textwidth][c]{
    \begin{tblr}{
    colspec = {ccc},
    colsep = 3pt,
    rowsep = 2.5pt,
    vlines
    }
    \hline
    \textbf{$g$} & \textbf{morph.} & \textbf{fiber description} \\
    \hline \hline
    \hyperref[sec:g2]{$2$} & $\mathcal{V}_0\to \mathcal{V}_1$ & $\mathbb{P}^1$ \\
    \hline
    \SetCell[r=2]{c} \hyperref[sec:g3]{$3$}
        & $\mathcal{V}_1 \to \mathcal{V}_2$ & $\alpha_1^{p+1} + \alpha_2^{p+1} + \alpha_3^{p+1} $ \\
        & $\mathcal{V}_0 \to \mathcal{V}_1$ & $\mathbb{P}^1$ \\
    \hline
    \SetCell[r=5]{c} \hyperref[sec:g4]{$4$}
        & $\mathcal{V}_2 \to \mathcal{V}_3$ & $\alpha_1 \alpha_2^{p^2} - \alpha_1^{p^2} \alpha_2 + \alpha_3 \alpha_4^{p^2} - \alpha_3^{p^2} \alpha_4$ \\
        & \SetCell[r=3]{c} $\mathcal{V}_1 \to \mathcal{V}_2$
            & $\alpha_1\alpha_6^p + \alpha_3\alpha_8^p- \alpha_4\alpha_7^p-\alpha_6\alpha_1^p\alpha_5^{p-1}+\alpha_7\alpha_4^p\alpha_5^{p-1} - \alpha_8\alpha_3^p\alpha_5^{p-1}$\\
            & & $\alpha_1\alpha_9^p + \alpha_3\alpha_{11}^p-\alpha_4\alpha_{10}^p$ \\
            & & $-\alpha_9\alpha_1 ^p + \alpha_{10}\alpha_4^p-\alpha_{11}\alpha_3^p$ \\
        & $\mathcal{V}_0 \to \mathcal{V}_1$ & $\mathbb{P}^1$ \\
    \hline
    \SetCell[r=8]{c} \hyperref[sec:g5]{$5$}
    & \SetCell[r=2]{c} $\mathcal{V}_3 \to \mathcal{V}_4$
            & $\alpha_1^{p^3+1} + \alpha_2^{p^3+1} + \alpha_3^{p^3+1} + \alpha_4^{p^3+1} + \alpha_5^{p^3+1}$ \\
            & & $\alpha_1^{p+1} + \alpha_2^{p+1} + \alpha_3^{p+1} + \alpha_4^{p+1} + \alpha_5^{p+1}$ \\
    & \SetCell[r=4]{c} $\mathcal{V}_2 \to \mathcal{V}_3$
            & $-\alpha_2\alpha_7^{p^2}-\alpha_3\alpha_8^{p^2}-\alpha_4\alpha_9^{p^2}-\alpha_5\alpha_{10}^{p^2} +\alpha_7\alpha_6^{p^2-1}\alpha_2^{p^2}+\alpha_8\alpha_6^{p^2-1}\alpha_3^{p^2}+\alpha_9\alpha_6^{p^2-1}\alpha_4^{p^2}+\alpha_{10}\alpha_6^{p^2-1}\alpha_5^{p^2}$ \\
            & & $-\alpha_2\alpha_{11}^{p^2}-\alpha_3\alpha_{12}^{p^2}-\alpha_4\alpha_{13}^{p^2}-\alpha_5\alpha_{14}^{p^2}$ \\
            & & $\alpha_{11}\alpha_2^{p^2}+\alpha_{12}\alpha_3^{p^2}+\alpha_{13}\alpha_4^{p^2}+\alpha_{14}\alpha_5^{p^2}$ \\
            & & $\alpha_2\alpha_{11} + \alpha_3\alpha_{12} + \alpha_4\alpha_{13} + \alpha_5\alpha_{14}$ \\
    & $\mathcal{V}_1 \to \mathcal{V}_2$ & see equation (\ref{eq:arb5})\\
    & $\mathcal{V}_0 \to \mathcal{V}_1$ & $\mathbb{P}^1$\\
    \hline
    \end{tblr}
    }
    \vspace{-10pt}
    \caption{Fibers of the truncation morphisms for arbitrary parameters.}
    \label{tab:arb}
    \vspace{-10pt}
\end{table}
\end{thm}

\begin{org}
    The article is organised as follows: In Section \ref{sec:AV}, we review the theory of supersingular abelian varieties. In particular, we introduce their moduli and Dieudonné theory as the main technical tool for computations. Section \ref{sec:PFTQ} gives a short overview over the framework of polarised flag type quotients by fixing the necessary definitions and summarising the main results established by Li and Oort. Section \ref{sec:StructureTHM} is the technical heart of this article: Here, we develop general principles for the calculations of PFTQs and prove the two structural results simplifying these computations. Finally, Section \ref{sec:calculations} shows the power of the developed theory by computing polarised flag type quotients explicitly in low dimensions.
\end{org}

\begin{ack}
    I would like to express my gratitude to my doctoral advisor Christian Liedtke for his introduction to this topic, his continued support, and for all the helpful conversations pertaining to this article.   I also thank Steven Groen, Valentijn Karemaker, and Eva Viehmann for their comments and feedback on the first version of this article. 
\end{ack}

\newpage

\section{Dieudonné theory of abelian varieties}\label{sec:AV}

In this section, we give a review of supersingular abelian varieties and their moduli following \parencite{Karemaker}. Furthermore, we introduce the technical tool translating their study into semilinear algebra, namely Dieudonné theory. 

\subsection{Supersingular abelian varieties}

Let $k$ be a field of characteristic $p$ and let $\overline{k}$ be its algebraic closure. We call an elliptic curve $E$ \emph{supersingular} if its group of $\overline{k}$-rational $p$-torsion points is trivial, i.e.\ if $E[p](\overline{k}) = \{0\}.$ There are two related notions in higher dimension $g \geq 2$:

\begin{df}
    Let $X$ be a $g$-dimensional abelian variety over $k$.
    \begin{itemize}
        \item[(i)] $X$ is called \emph{supersingular} if it is isogenous to a product of $g$ supersingular elliptic curves over its algebraic closure: $X \simeq_{\overline{k}} E_1 \times \ldots \times E_g.$
        \item[(ii)] $X$ is called \emph{superspecial} if it is isomorphic to a product of $g$ supersingular elliptic curves over its algebraic closure: $X \cong_{\overline{k}} E_1 \times \ldots \times E_g.$
    \end{itemize}
\end{df}

\begin{rem}
    By a result due to Deligne, Ogus, and Shioda \parencite[Theorem 3.5]{Shioda}, for any supersingular elliptic curves $E_1, \ldots, E_{2g}$ with $g \geq 2$, we have
    $$E_1 \times \ldots \times E_g \cong E_{g+1} \times \ldots \times E_{2g}.$$
    In particular, all supersingular (resp.\ superspecial) abelian varieties of dimension $g\geq2$ are geometrically isogenous (resp.\ isomorphic), and the above definition is independent of the choice of supersingular elliptic curves. 
\end{rem}

As all supersingular abelian varieties are of $p$-rank $0$, this invariant is not useful for analysing these objects. However, the $a$-number will often be used:

\begin{df}
    Let $X$ be an abelian variety over $k$. Its \emph{$a$-number} is defined to be $$a(X) = \text{dim}_k(\text{Hom}(\alpha_p,X)).$$
    A supersingular abelian variety $X$ over $k$ is called \emph{supergeneral} if $a(X) = 1$. 
\end{df}

For an abelian variety $X$, we denote its dual by $X^{\vee}$. Given a line bundle $\mathcal{L}$ on $X$, we define $\varphi_{\mathcal{L}}: X \to X^{\vee}$ to be the homomorphism given by $x \mapsto [t_x^*\mathcal{L} \otimes \mathcal{L}^{-1}]$ on points, where the square bracket denotes the equivalence class of the line bundle. 

\begin{df}
    A \emph{polarisation} of $X$ is an isogeny $\lambda: X \to X^{\vee}$ such that there exists a finite field extension $k'\supseteq k$ and an ample line bundle $\mathcal{L}$ on $X \times k'$ such that $\lambda_{k'} = \varphi_{\mathcal{L}}$. We call a polarisation \emph{principal} if it is an isomorphism.
\end{df}

Let $\mathcal{A}_g$ be the moduli space of $g$-dimensional principally polarised abelian varieties. We denote its supersingular locus by $\mathcal{S}_g$, i.e.\ $$ \mathcal{S}_g = \{(X, \lambda) \in \mathcal{A}_g: X \text{ is supersingular}\}.$$ Note that $\mathcal{S}_g$ is a Zariski closed subset of $\mathcal{A}_g$ which can be given an induced reduced subscheme structure. Furthermore, it is a coarse moduli space for supersingular principally polarised abelian varieties. We refer to \parencite[Theorems 7.9 and 7.10]{Mumford} for GIT-theoretic proofs of existence of these spaces.

\subsection{Dieudonné theory}

Let $k$ be a perfect field of characteristic $p$ and let $W(k)$ be the ring of Witt vectors over $k$. We denote the Frobenius on $W(k)$ by $\sigma$ and the Teichmüller lift of an element $a\in k$ to $W(k)$ by $\underline{a} = (a, 0, \ldots)$.

\begin{df}
    The \emph{Dieudonné ring} $\mathcal{D}_k$ is the noncommutative ring over $W(k)$ generated by two indeterminates $F$ and $V$ subject to the relations
    {\small\begin{align*}
        Fw = \sigma(w)F \;\; \forall\, w \in W(k), \quad
        wV = V\sigma(w) \;\; \forall\, w \in W(k), \quad
        FV = VF = p.
    \end{align*}}
    A left module over $\mathcal{D}_k$ is called a \emph{Dieudonné module}.
\end{df}

The central theorem of Dieudonné theory is the following correspondence proven in this form in \parencite[Chapter III.8]{Demazure} (see also \parencite[Chapter V]{DemazureGabriel}).
\begin{thm}\label{thm:DieudonneBarsottiTate}
    There exists an antiequivalence $M$ from the category of Barsotti-Tate groups over $k$ to the category of Dieudonné modules that are free and of finite rank over $W(k)$. Moreover, the height of the Barsotti-Tate group is equal to the rank of the Dieudonné module and the Dieudonné module of the Serre dual $G^t$ is given by $M(G^t) = \mathrm{Hom}_{W(k)}(M(G), W(k)).$
\end{thm}

Now, let $X$ be a $g$-dimensional abelian variety over $k$. Recall that the $n$-torsion $X[n]$ of $X$ is defined to be the kernel of the multiplication-by-$n$ map and is a finite group scheme of rank $n^{2g}$. 
\begin{df}
    Let $X$ be an abelian variety over $k$. Then, the \emph{$p$-divisible group} associated to $X$ is defined to be the direct limit of group schemes
    \begin{align*}
        X[p^{\infty}] = \varinjlim X[p^n],
    \end{align*}
    where the transition maps are given by the inclusions $X[p^n]\hookrightarrow X[p^{n+1}]$.
\end{df}

As the $p$-divisible group associated to an abelian variety is a Barsotti-Tate group of height $2g$ (see \parencite[Example 2.1]{TatePDiv}), we may apply the Dieudonné correspondence to this group, motivating the following definition:

\begin{df}
    Let $X$ be an abelian variety over $k$. Its \emph{Dieudonné module} is defined to be $$M(X) = M(X[p^{\infty}]).$$
\end{df}

Similar to abelian varieties, we define invariants for Dieudonné modules: 

\begin{df}\label{df:anumber}
Let $M$ be a Dieudonné module of finite rank over $W(k)$.
\begin{itemize}
    \item[(i)] The \emph{$a$-number} of $M$ is defined as $a(M) = \dim_k(M/(F,V)M).$
    \item[(ii)] The \emph{genus} of $M$ is given by $g(M) = \frac{1}{2} \text{rank}_{W(k)}(M).$
\end{itemize}    
\end{df}

\begin{pro}\label{pro:anumber}
    Let $X$ be an abelian variety over $k$. 
    \begin{itemize}
        \item[(i)] The $a$-numbers of $X$ and its Dieudonné module coincide:
        \begin{align*}
            a(M(X)) = a(X).
        \end{align*}
        \item[(ii)] The dimension of $X$ agrees with the genus of its Dieudonné module:
        \begin{align*}
            g(M(X)) = \dim(X).
        \end{align*}
    \end{itemize}
\end{pro}

\begin{proof}
    (ii) follows from Theorem \ref{thm:DieudonneBarsottiTate}. For (i), see \parencite[Chapter 2.5]{LiOort}.
\end{proof}

Having introduced these invariants, it is not surprising that many further notions related to abelian varieties translate nicely to the Dieudonné world:

\begin{df}
A Dieudonné module $M$ over $W(k)$ is called \emph{supersingular} (resp.\ \emph{superspecial}, \emph{supergeneral}) if there exists an abelian variety $X$ over $\overline{k}$ which is supersingular (resp.\ superspecial, supergeneral) such that
\begin{align*}
    M \otimes_{W(k)} W(\overline{k}) \cong M(X).
\end{align*}
\end{df}


In view of Theorem \ref{thm:DieudonneBarsottiTate}, we adopt the following definition of dual modules:

\begin{df}
    Let $M$ be a Dieudonné module that is a free $W(k)$-module. The \emph{dual Dieudonné module} $M^t$ associated to the primal module $M$ is defined to be
    \begin{align*}
        M^t = \text{Hom}_{W(k)}(M, W(k))
    \end{align*}
    considered as a $W(k)$-module. Furthermore for $f \in M^t$, we define
    \begin{align*}
        (Ff)(m) = f(\sigma(Vm)), \quad (Vf)(m) = f(\sigma^{-1}(Fm)) \quad \text{for all } m\in M.
    \end{align*}
\end{df}

\subsection{Quasi-polarisations}

From the categorical equivalence, it follows that a polarisation $\eta: X \to X^{\vee}$ on an abelian variety induces a morphism between the associated Dieudonné modules. Via the canonical isomorphism $M(X^{\vee}) \cong M(X)^t$, we obtain the monomorphism $M(\eta): M(X)^t \to M(X).$ Composing this map with the canonical pairing yields the following form:
\begin{df}
    Let $X$ be a polarised abelian variety over a perfect field $k$. We define the \emph{quasi-polarisation} of $X$ on $M(X)^t$ to be
    \begin{align*}
        \langle\cdot,\cdot\rangle: M(X)^t \times M(X)^t \to W(k), \quad
        (a,b) \mapsto b(M(\eta)(a)).
    \end{align*}    
\end{df}

The quasi-polarisation encodes the geometry of the polarisation in the language of Dieudonné theory. This manifests in the following semilinear algebraic properties as shown in \parencite[Proposition 3.24]{Oda}:

\begin{pro}\label{pro:quasipolarisationdual}
Let $X$ be a polarised abelian variety over a perfect field $k$. Then, the quasi-polarisation of $X$ on $M(X)^t$ is a non-degenerate alternating $W(k)$-bilinear form such that for all $x,y \in M(X)^t$
\begin{align*}
    \langle Fx, y \rangle = \sigma(\langle x, Vy\rangle).
\end{align*}
\end{pro}

To extend to the primal Dieudonné module, we use the following lemma:

\begin{lem}\label{lem:embedding}
    Let $X$ be a polarised abelian variety over a perfect field $k$. Then, the polarisation induces
    \begin{align*}
        M(X) \xhookrightarrow{\iota} M(X)^t\otimes_{W(k)}\mathrm{Frac}(W(k)).
    \end{align*}
\end{lem}

\begin{proof}
    Since the polarisation $\eta: X \to X^{\vee}$ is an isogeny of abelian varieties, the corresponding Dieudonné modules become isomorphic over $\text{Frac}(W(k))$ by \parencite[Chapter IV.1]{Demazure}. Via the natural inclusion of a module into its base extension, we obtain
    \begin{align*}
        M(X) \hookrightarrow M(X) \otimes_{W(k)} \text{Frac}(W(k)) &\cong M(X^{\vee})\otimes_{W(k)}\text{Frac}(W(k)) \\
        &\cong M(X)^t\otimes_{W(k)} \text{Frac}(W(k)). \qedhere
    \end{align*}
\end{proof}

This embedding enables us to extend the quasi-polarisation from the dual to the primal Dieudonné module at the cost of enlarging the codomain:

\begin{thm}\label{thm:QuasipolarisationFull}
    Let $X$ be a polarised abelian variety over a perfect field $k$. Then, the polarisation induces a non-degenerate alternating $W(k)$-bilinear form on $M(X)$ 
    \vspace{-4pt}
    \begin{align*}
        \langle\cdot, \cdot\rangle: M(X) \times M(X) \to \mathrm{Frac}(W(k)),
    \end{align*}
    also called the quasi-polarisation. The form is given by composing the embedding $\iota$ with the quasi-polarisation on the dual Dieudonné module. Moreover:
    \begin{itemize}
        \item[(i)] For all $x,y \in M(X)$, the following compatibility condition holds: 
        \begin{align*}
            \langle Fx,y\rangle = \sigma(\langle x, Vy\rangle).
        \end{align*}
        \item[(ii)] In the embedding $M(X)^t \subseteq M(X)$, the dual module can be identified:
        \begin{align*}
            x \in M(X)^t \iff \langle x, M(X)\rangle \subseteq W(k).
        \end{align*}
    \end{itemize}
\end{thm}

\begin{proof}
    The construction of the quasi-polarisation as well as the first property follow immediately from Lemma \ref{lem:embedding} and Proposition \ref{pro:quasipolarisationdual}.

    Property (ii) follows from a general result on dual modules shown in \parencite[Theorem 2.7]{Conrad}: For an integral domain $R$ with fraction field $K$ and a nonzero $R$-module $M$ in $K$, the dual is given as $M^t \cong \{x \in K \!\mid\! xM \!\subseteq \!R\}.$
\end{proof}

Motivated by the previous theorem, we introduce the notion of a Dieudonné module endowed with the additional structure of such a bilinear form:

\begin{df}
A \emph{quasi-polarised Dieudonné module} is a Dieudonné module $M$ equipped with a non-degenerate alternating $W(k)$-bilinear form
\begin{align*}
    \langle \cdot, \cdot \rangle: M \times M \to \text{Frac}(W(k))
\end{align*}
satisfying condition (i) for the Frobenius and Verschiebung maps.
\end{df}

We conclude this review by recalling a structure theorem for quasi-polarised superspecial Dieudonné modules first proven in \parencite[Chapter 6.1]{LiOort}.

\begin{thm}\label{thm:PreliminaryStructure}
    Let $M$ be a quasi-polarised superspecial Dieudonné module of genus $g$ over $W(k)$ such that $M \cong A_{1,1}^{\oplus g}$. Then, $M$ decomposes as
    \begin{align*}
        M \cong M_1 \oplus M_2 \oplus \ldots \oplus M_d,
    \end{align*}
    where $\langle M_i, M_j\rangle = 0$ for $i\neq j$, and each $M_i$ is of one of the following types:
    \begin{itemize}
        \item[(i)] a genus-$1$ quasi-polarised superspecial Dieudonné module over $W(k)$ generated by $x$ as a Dieudonné module such that
        \begin{align*}
            \langle x, Fx\rangle = p^r\epsilon
        \end{align*}
        for some $r\in\mathbb{Z}$ and $\epsilon \in W(k) \setminus pW(k)$ satisfying $\sigma(\epsilon) = -\epsilon$, or
        \item[(ii)] a genus-$2$ quasi-polarised superspecial Dieudonné module over $W(k)$ generated by $x,y$ as a Dieudonné module such that
        \begin{gather*}
            \langle x,y\rangle = p^r, \\
            \langle x,Fx\rangle = \langle y,Fy \rangle = \langle x, Fy\rangle = \langle y, Fx\rangle = 0
        \end{gather*}
        for some $r \in \mathbb{Z}.$
    \end{itemize}
\end{thm}

\section{Polarised flag type quotients}\label{sec:PFTQ}
Aiming to describe the moduli space of supersingular abelian varieties more explicitly, the previous theorem provides a crucial first step by resolving the superspecial case: Assuming that the base field is algebraically closed, $X$ corresponds to a superspecial Dieudonné module with a quasi-polarisation, and we have established a structure theorem classifying such objects. 

Building on this, it is natural to ask whether supersingular abelian varieties can be linked to superspecial ones in a way that allows us to leverage this classification. This motivates the framework of flag type quotients. 

The key insight behind (polarised) flag type quotients is that any supersingular abelian variety can be connected to a superspecial abelian variety via a purely inseparable isogeny. As the kernel of this isogeny is constructed from successive extensions of the group scheme $\alpha_p$, the isogeny can be factored into a chain of isogenies, each having a predetermined kernel rank.

\subsection{PFTQs of abelian varieties}

Before stating the definition, we fix some notation assuming that $k$ is an algebraically closed field of characteristic $p$. It is well-known that for every $p$, there exists a supersingular elliptic curve over $\mathbb{F}_p$ whose Frobenius satisfies $F^2 +p =0$, see, for example \parencite[Chapter 1.2]{LiOort}; we denote such a curve by $E_0$. Furthermore, let $\eta$ be a polarisation on $E_0^g \otimes k$ with kernel
\begin{equation*}
    \text{ker}(\eta) = \text{ker}(F^{g-1}) \otimes k.
\label{eq:polarisation}
\tag{1}
\end{equation*}
Note that such a polarisation already exists for $\mathbb{F}_{p^2}$, see \parencite[Chapter 3.6]{LiOort}.

\begin{df}
    An \emph{$\alpha$-group} $G$ of $\alpha$-rank $r$ is a finite flat group scheme over an $\mathbb{F}_p$-scheme $S$ such that both Frobenius and Verschiebung vanish, i.e.\ $F_{G} = 0$ and $V_G = 0$. In particular, it is locally isomorphic to $\alpha_p^r \times S$. 
\end{df}

\begin{df}
    Let $g \geq 1$, $S$ be an $\mathbb{F}_{p^2}$-scheme, and $E_0$ the fixed supersingular elliptic curve. A \emph{flag type quotient} (FTQ) is a sequence of $g$-dimensional supersingular abelian schemes over $S$ of the form
    \begin{align*}
        Y_{g-1} \xrightarrow{\rho_{g-1}}Y_{g-2} \xrightarrow{\rho_{g-2}} \ldots \xrightarrow{\rho_2}Y_1 \xrightarrow{\rho_1}Y_0,
    \end{align*}
    satisfying the following conditions:
    \begin{itemize}
        \item[(i)] The top step is a superspecial abelian scheme $Y_{g-1} = E_0^g \times S$, and
        \item[(ii)] $\text{ker}(\rho_i)$ is an $\alpha$-group of $\alpha$-rank $i$ for all $1 \leq i \leq g-1$. 
    \end{itemize}
    Furthermore, we call it a \emph{polarised flag type quotient} (PFTQ) with respect to the polarisation $\eta$ if the following additional conditions are satisfied:
    \begin{itemize}
        \item[(iii)] There are polarisations $\eta_i$ on $Y_i$ for all $0 \leq i \leq g-1$ that make the following diagram commute:
        \begin{equation*}
            \begin{tikzcd}
                Y_{g-1} \arrow[r, "\rho_{g-1}"] \arrow[d, "\eta_{g-1}"] & Y_{g-2} \arrow[r, "\rho_{g-2}"] \arrow[d, "\eta_{g-2}"] & \cdots \arrow[r, "\rho_2"] & Y_1 \arrow[r, "\rho_1"] \arrow[d, "\eta_{1}"] & Y_0 \arrow[d, "\eta_0"] \\
                Y_{g-1}^{\vee} & Y_{g-2}^{\vee} \arrow[l, "\rho_{g-1}^{\vee}"] & \cdots \arrow[l, "\rho_{g-2}^{\vee}"] & Y_{1}^{\vee} \arrow[l, "\rho_2^{\vee}"] & Y_{0}^{\vee} \arrow[l, "\rho_1^{\vee}"]
            \end{tikzcd}
        \end{equation*}
        \item[(iv)] The top-step polarisation is given by $\eta_{g-1} = \eta \times \text{id}_S$.
        \item[(v)] The following descent condition holds for all $0 \leq i \leq g-1$:
        \begin{align*}
            \text{ker}(\eta_i) \subseteq Y_i[(F_{Y_i})^{i-j} \circ (V_{Y_i})^j] \quad \forall \ 0 \leq j \leq \left\lfloor \tfrac{i}{2} \right\rfloor.
        \end{align*}
    \end{itemize}
    Moreover, a (polarised) flag type quotient is called \emph{rigid} if additionally
    \begin{itemize}
        \item[(vi)] the following rigidity condition holds for all $1 \leq i \leq g-1$:
        \begin{align*}
        \text{ker}(Y_{g-1} \to Y_{i}) = \text{ker}(Y_{g-1}\to Y_0) \cap Y_{g-1}[F^{g-1-i}].
        \end{align*}
    \end{itemize}
\end{df}

It was shown in \parencite{LiOort} that rigid PFTQs can be used to describe the moduli space of principally polarised supersingular abelian varieties. We briefly recall the two main theoretical theorems formalising this connection:

\begin{lem}
    Let $\eta$ be a polarisation satisfying the kernel requirement (\ref{eq:polarisation}). Then, the functor
    \begin{align*}
        \mathfrak{p}_{g,\eta}: \textbf{Sch}_k&\to \textbf{Set} \\
        S &\mapsto \{\text{PFTQs over $S$ with respect to $\eta$}\}\big/\text{isomorphism}
    \end{align*}
    is represented by a projective scheme $\mathcal{P}_{g, \eta}$ over $k$. Similarly, the functor associated to rigid PFTQs is represented by a quasi-projective scheme $\mathcal{P}_{g, \eta}'$ over $k$.
\end{lem}

For later computations, it is convenient to consider partial flag structures. To this end, we define \emph{$m$-truncated PFTQs} $$\{(Y_i, \eta_i)_{(m \leq i \leq g-1)};{\rho_i}_{(m+1 \leq i \leq g-1)}\}$$ as sequences of abelian schemes satisfying the above PFTQ conditions within the restricted index range. As a corollary of the previous lemma, $m$-truncated PFTQs also admit fine moduli spaces, denoted by $\mathcal{V}_m$, which fit into a natural tower of so-called \emph{truncation morphisms}:
\begin{align*}
    \mathcal{P}_{g, \eta} = \mathcal{V}_0 \xrightarrow{} \mathcal{V}_1 \xrightarrow{} \ldots \xrightarrow{} \mathcal{V}_{g-2} \xrightarrow{} \mathcal{V}_{g-1} = \{\text{one point}\}.
\end{align*}

Crucially, the moduli spaces of abelian varieties and PFTQs are closely linked as it can be shown that every abelian variety is at the end of at least one and at most finitely many rigid PFTQs. This observation is formalised in the following proposition: 

\begin{pro}\label{pro:surjectivequasifiniteMorphism}
    Let $k = \overline{\mathbb{F}}_p$, and let $\Lambda$ be a set of representatives of equivalence classes of polarisations satisfying the kernel requirement (\ref{eq:polarisation}). Then, the morphism
    \begin{align*}
        \coprod_{\eta\in\Lambda} \mathcal{P}_{g,\eta}' \to \mathcal{S}_g \otimes \overline{\mathbb{F}}_p
    \end{align*}
    projecting a PFTQ to its last member is surjective and quasi-finite.
\end{pro}

\subsection{PFTQs of Dieudonné modules}
To explicitly work with PFTQs of abelian varieties, it is helpful to reformulate their definition in the language of Dieudonné theory which leads to the following notion: 

\begin{df}\label{df:PFTQDieudonne}
A \emph{polarised flag type quotient of Dieudonné modules} of genus $g$ over $W(k)$ is a filtration
\begin{align*}
    M_{g-1} \supseteq M_{g-2} \supseteq \ldots \supseteq M_1 \supseteq M_0
\end{align*}
of quasi-polarised supersingular Dieudonné modules of genus $g$ satisfying the following properties:
\begin{itemize}
    \item[(i)] $M_{g-1} \cong A_{1,1}^{\oplus g}$ with $M_{g-1}^t = F^{g-1}M_{g-1}$,
    \item[(ii)] $(F,V)M_i \subseteq M_{i-1}$ and $\text{dim}_k(M_i/M_{i-1}) = i$ for all $1\leq i \leq g-1$,
    \item[(iii)] $F^{i-j}V^{j}M_i \subseteq M_i^t$ for all $0 \leq j \leq \left\lfloor \tfrac{i}{2} \right\rfloor$. 
\end{itemize}
Moreover, we call the PFTQ \emph{rigid} if the following additional condition holds:
\begin{itemize}
    \item[(iv)] $M_i = M_0 + F^{g-1-i}M_{g-1}$ for all $1\leq i \leq g-1$.  
\end{itemize}
\end{df}

\begin{rem}
    Note that $A_{1,1} = \mathcal{D}_k/\mathcal{D}_k(F-V)$ is the Dieudonné module of the fixed elliptic curve $E_0$ after base extension to a field containing $\mathbb{F}_{p^4}$.
\end{rem}

Clearly, PFTQs of abelian varieties and of Dieudonné modules are closely linked, as formalised by the following correspondence in \parencite[Lemma 7.4]{LiOort}:

\begin{pro}
The Dieudonné correspondence gives rise to a bijection
\begin{align*}
    \{\text{PFTQs of abelian varieties with respect to $\eta$ over $k$}\} \big/ \text{isomorphisms} \\
    \overset{1:1}{\longleftrightarrow}\{\text{PFTQs of Dieudonné modules of genus $g$ over $W(k)$}\},
\end{align*}
in which rigid PFTQs of abelian varieties correspond to rigid PFTQs of Dieudonné modules.
\end{pro}

In summary, Proposition \ref{pro:surjectivequasifiniteMorphism} shows that polarised flag type quotients provide a useful framework for studying the moduli space of principally polarised supersingular abelian varieties. In this section, we have seen that we can also pass to Dieudonné modules and analyse their PFTQs.

To finish this section, we collect some facts about PFTQs of Dieudonné modules that are helpful for later calculations:

\begin{lem}\label{lem:PFTQLemma}
Let $M_{g-1} \supseteq \ldots \supseteq M_0$ be a PFTQ of Dieudonné modules of genus $g$ over $W(k)$. Then, the following statements hold for all $i$:
\begin{itemize}
    \item[(i)] Each successive quotient $M_i/M_{i-1}$ is a $k$-vector space.
    \item[(ii)] The genus can be recovered from the relation $\dim_k(M_i/FM_i) = g.$ 
\end{itemize}
\end{lem}

\begin{proof}
    For (i), it follows from property (ii) in the definition of PFTQs of Dieudonné modules that $p = F \circ V$ acts as zero on $M_i/M_{i-1}$. 
    This implies that $M_i/M_{i-1}$ is a module over $W(k)/(p) \cong k$, and thus a $k$-vector space.

    For (ii), $M(G)/FM(G)$ is isomorphic to the dual of the tangent space of $G$ at the identity by a corollary of the Dieudonné Theorem \ref{thm:DieudonneBarsottiTate}, see also \parencite[Proposition 4.4]{Fontaine}.
    As the dimension of the cotangent space is equal to the dimension of the $p$-divisible group, the stated dimension formula follows.
\end{proof}

\clearpage

\section{Structure theorems}\label{sec:StructureTHM}

In low-dimensional examples, one naturally computes PFTQs by working downwards from the top level. The structure theorems in this section formalise this principle: We first fix the top level of a PFTQ by a normal form and then provide an easily verifiable descent criterion for passing to the next lower level. Together, these results provide a computational framework for calculations of PFTQs and suffice to describe the first and last step.

Throughout, we assume that $k$ is algebraically closed and contains $\mathbb{F}_p$.

\subsection{Normal form at top level of PFTQ}
The additional constraints imposed by the PFTQ structure allow us to prove a refined version of the structure theorem for quasi-polarised superspecial Dieudonné modules:

\begin{thm}\label{thm:Structure}
Let $M_{g-1}$ be the quasi-polarised superspecial Dieudonné module at the top of a PFTQ.
\begin{itemize}
    \item[(i)] If $g$ is odd, there is a basis $e_1, f_1, \ldots, e_g,f_g$ of $M_{g-1}$ as a free $W(k)$-module such that 
    \begin{align*}
        Fe_i = f_i, \quad Ve_i = -f_i,\quad Ff_i = -pe_i,\quad Vf_i = pe_i,
    \end{align*}
    and the quasi-polarisation $M_{g-1}^t \times M_{g-1}^t \to W(k)$ with respect to the dual basis is given by the block diagonal matrix
    \begin{align*}
    \begin{pmatrix}
        0 & \phantom{-} p^{\frac{g-1}{2}} \\ -p^{\frac{g-1}{2}} & 0
    \end{pmatrix}^{\bigoplus g}.
    \end{align*}
    \item[(ii)] If $g$ is even, there is a basis $e_1, f_1, \ldots, e_g,f_g$ of $M_{g-1}$ as a free $W(k)$-module such that
    \begin{align*}
        Fe_i = f_i,\quad Ve_i = f_i,\quad Ff_i = pe_i,\quad Vf_i = pe_i,
    \end{align*}
    and the quasi-polarisation $M_{g-1}^t \times M_{g-1}^t \to W(k)$ with respect to the dual basis is given by the block diagonal matrix
    {\begin{align*}\begin{pNiceMatrix}[c, margin=2pt]
    \Block{2-2}<\Large>{0} 
    & & \phantom{-}p^{\frac{g}{2}} & 0 \\
    & & 0 & \phantom{-}p^{\frac{g}{2}-1} \\
    -p^{\frac{g}{2}} & 0 & \Block{2-2}<\Large>{0}&  \\
    0 & -p^{\frac{g}{2}-1} &  & \\
    \CodeAfter
    \begin{tikzpicture}
        \draw[line width = 0.2pt, dashed] (3-|1) -- (3-|5) ;
        \draw[line width = 0.2pt, dashed] (1-|3) |- (5-|3) ;
    \end{tikzpicture}   
    \end{pNiceMatrix}^{\bigoplus \frac{g}{2}}.
    \end{align*}}
\end{itemize}
\end{thm}

\begin{proof}
To avoid confusion with the Dieudonné modules appearing in the decomposition, we denote $M_{g-1}$ by $M$ throughout this proof.

The Dieudonné correspondence applied to a polarisation $\eta: X \to X^{\vee}$ implies $M(X)^t = \text{im}(M(\eta))$. Therefore, property (i) in the definition of a PFTQ can be reformulated as \begin{align*}\text{im}(F^{g-1}) = \text{im}(M(\eta_{g-1})).\tag{i'}\end{align*}
Since $M$ is superspecial, Theorem \ref{thm:PreliminaryStructure} yields a decomposition $$M = M_1 \oplus \ldots \oplus M_d,$$ where each $M_i$ is generated by either one element $x_i$ or by two elements $x_i,y_i$ as a Dieudonné module. Relabelling the resulting $g$ generators as $\{e_i\}_{i=1}^g$ and defining $f_i := Fe_i$, the set $\{e_i, f_i\}_{i = 1}^g$ forms a $W(k)$-basis of $M$.

With these two preparations, we begin the proof of the structure theorem for even genus. To deduce the claimed statement from the properties established in Theorem \ref{thm:PreliminaryStructure}, we verify the following two claims:
\begin{itemize}
    \item[(1)] There is no genus-$1$ component in the decomposition of $M$, and
    \item[(2)] the free parameter is given by $r = \frac{g}{2}$.
\end{itemize}

For (1), we assume that there is a genus-$1$ component, and without loss of generality, it is the first one. We show that this contradicts the reformulated property (i') stated above. To this end, we firstly calculate $\text{im}(F^{g-1})$ using the following recursive relations:
\begin{align*}
    F^{g-1}(e_i) &= F^{g-2}(f_i) = F^{g-3}(pe_i) = pF^{g-3}(e_i) = \ldots = p^{\frac{g}{2}-1}f_i, \\
    F^{g-1}(f_i) &= F^{g-2}(pe_i) = pF^{g-2}(e_i) = pF^{g-3}(f_i) = \ldots = p^{\frac{g}{2}}e_i.
\end{align*} 
It follows that $\text{im}(F^{g-1}) = \langle p^{\frac{g}{2}}e_1, p^{\frac{g}{2}-1}f_1, \ldots, p^{\frac{g}{2}}e_g, p^{\frac{g}{2}-1}f_g\rangle$.

\noindent Next, consider $\text{im}(M(\eta_{g-1}))$, calculating the contribution of the genus-$1$ component: 
\begin{align*}
    M(\eta_{g-1})(e_1^*) &= -p^r\epsilon f_1, \\
    M(\eta_{g-1})(f_1^*) &= p^r\epsilon e_1.
\end{align*}
As the quasi-polarisation is of block-diagonal shape, there are no further possibilities to obtain $e_1$ and $f_1$ in the image of $M(\eta_{g-1})$. The reformulated property (i') implies that $\text{im}(F^{g-1}) \cap \langle e_1, f_1 \rangle = \text{im}(M(\eta_{g-1}))\cap \langle e_1, f_1\rangle$, and hence we obtain
\begin{align*}
    \langle p^{\frac{g}{2}}e_1, p^{\frac{g}{2}-1}f_1\rangle = \langle p^r\epsilon e_1, p^r\epsilon f_1\rangle. 
\end{align*}
However, this cannot hold for any $r \in \mathbb{Z}$ as $p$ is not invertible in the Witt vectors. Therefore, only genus-$2$ components appear in the decomposition of $M$, proving claim (1).
\vspace{4pt}

For (2), we proceed as before, now knowing that all components are of genus $2$. Since the image of $F^{g-1}$ is independent of the quasi-polarisation, we can use the result calculated above. Calculating the image of the polarisation, we obtain for the first genus-$2$ block that
\begin{alignat*}{4}
    M(\eta_{g-1})(e_1^*) &= -p^re_2, &\quad\quad M(\eta_{g-1})(e_2^*) &= p^re_1,\\
    M(\eta_{g-1})(f_1^*) &= -p^{r-1}f_2, &\quad\quad M(\eta_{g-1})(f_2^*) &= p^{r-1}f_1,
\end{alignat*}
and therefore $\text{im}(M(\eta_{g-1})) = \langle p^re_1, p^{r-1}f_1, \ldots, p^re_g, p^{r-1}f_g\rangle$. 
Comparing images, equality in the reformulated property (i') holds for $r = \frac{g}{2}$.
\vspace{8pt}

We now assume that $g$ is odd. With the same calculations as in the even genus case,  namely comparing $\text{im}(F^{g-1})$ and $\text{im}(M(\eta_{g-1}))$, one verifies that
\begin{itemize}
    \item[(1)] there is no genus-$2$ component in the decomposition of $M$, and
    \item[(2)] the free parameter is given by $r = \frac{g-1}{2}$.
\end{itemize}
\vspace{2pt}

However, this is only an intermediate step, as Theorem \ref{thm:PreliminaryStructure} with these two properties gives a skeletal basis $\{e_i, f_i\}$ such that the quasi-polarisation on $M^t$ with respect to the dual basis is given by
{\small\begin{align*}
    \begin{pmatrix} 0 & p^{\frac{g-1}{2}}\epsilon \\ p^{\frac{g-1}{2}}\sigma(\epsilon) & 0 \end{pmatrix}^{\oplus g}.
\end{align*}}

To obtain the claimed form, we perform a base change. For this, recall that there exists an element $\lambda \in W(\mathbb{F}_{p^4}) \setminus pW(\mathbb{F}_{p^4})$ such that $\sigma^2(\lambda) = -\lambda$. Defining $\epsilon = \lambda \cdot \sigma(\lambda)$, we obtain an element satisfying $\sigma(\epsilon) = -\epsilon$. In the proof of Theorem \ref{thm:PreliminaryStructure}, a slightly stronger statement is shown, namely that $\epsilon$ can be fixed in advance provided it satisfies this Frobenius relation. We use this freedom to fix our choice of $\epsilon$ to be the one defined above.

We now perform a base change by rescaling the skeletal basis as follows:
\begin{align*}
    e_i \mapsto \lambda e_i =:e_i', \quad\quad f_i \mapsto \sigma(\lambda)f_i =:f_i'.
\end{align*}

\noindent It is straightforward to verify that this basis satisfies the claimed relations:
\begin{alignat*}{4}
    Fe_i' &= \sigma(\lambda)Fe_i = \sigma(\lambda)f_i = f_i', &\quad
    Ff_i' &= \sigma^2(\lambda)Ff_i = - \lambda pe_i = -pe_i', \\
    Ve_i' &= \sigma^{-1}(\lambda)Ve_i = -\sigma(\lambda)f_i = -f_i', &\quad
    Vf_i' &= \sigma^{-1}(\sigma(\lambda))Vf_i = \lambda pe_i = pe_i'.
\end{alignat*}

\noindent As the dual basis transforms as $(e_i')^* = \frac{1}{\lambda}e_i^*$ and $(f_i')^* = \frac{1}{\sigma(\lambda)}f_i^*$, it holds that
\begin{align*}
    \langle (e_i')^*,(f_i')^*\rangle =  \frac{1}{\lambda\sigma(\lambda)}\langle e_i^*, f_i^*\rangle = \frac{1}{\lambda\sigma(\lambda)} p^{\frac{g-1}{2}}\epsilon = \frac{\lambda\sigma(\lambda)}{\lambda\sigma(\lambda)}p^{\frac{g-1}{2}} = p^{\frac{g-1}{2}}.
\end{align*}

\noindent In summary, this base change transforms the skeletal basis into one satisfying the wanted relations and giving the claimed matrix form, completing the proof of the odd genus case.
\end{proof}

This quasi-polarisation can be extended to the full Dieudonné module by applying the general construction described before Theorem \ref{thm:QuasipolarisationFull}:
\begin{cor}\label{cor:Structure}
Let $M_{g-1}$ be the quasi-polarised superspecial Dieudonné module at the top of a PFTQ. 
\begin{itemize}
    \item[(i)] If $g$ is odd, then the quasi-polarisation $M_{g-1} \times M_{g-1} \to W(k)[p^{-1}]$ with respect to the basis $\{e_1, f_1, \ldots, e_g, f_g\}$ of the previous theorem is given by
    $$\begin{pmatrix}
        0 & \phantom{-} p^{-\frac{g-1}{2}} \\
        -p^{-\frac{g-1}{2}} & 0
    \end{pmatrix}^{\bigoplus g}.$$
    
    \item[(ii)] If $g$ is even, then the quasi-polarisation $M_{g-1}\times M_{g-1} \to W(k)[p^{-1}]$ with respect to the basis $\{e_1, f_1, \ldots, e_g, f_g\}$ of the previous theorem is given by
    $$\begin{pNiceMatrix}[c, margin=2pt]
    \Block{2-2}<\Large>{0} 
    & & \phantom{-}p^{-\frac{g}{2}} & 0 \\
    & & 0 & \phantom{-}p^{-\frac{g}{2}+1} \\
    -p^{-\frac{g}{2}} & 0 & \Block{2-2}<\Large>{0}&  \\
    0 & -p^{-\frac{g}{2}+1} &  & \\
    \CodeAfter
    \begin{tikzpicture}
        \draw[line width = 0.2pt, dashed] (3-|1) -- (3-|5) ;
        \draw[line width = 0.2pt, dashed] (1-|3) |- (5-|3) ;
    \end{tikzpicture}
    \end{pNiceMatrix}^{\bigoplus \frac{g}{2}}.$$
\end{itemize}
\end{cor}

\begin{proof}
In the odd genus case, it follows from the previous theorem that $$M(\eta_{g-1})(e_i^*) = -p^{\frac{g-1}{2}}f_i, \quad \quad M(\eta_{g-1})(f_i^*) = p^{\frac{g-1}{2}}e_i.$$ Tensoring with $\mathrm{Frac}(W(k)) = W(k)[p^{-1}]$, we recover the primal basis:
$$e_i = M(\eta_{g-1})(f_i^*) \otimes p^{-\frac{g-1}{2}}, \quad \quad f_i = M(\eta_{g-1})(e_i^*) \otimes -p^{-\frac{g-1}{2}}.$$
From this, it is obvious that, up to a factor, the quasi-polarisation remains unchanged when passing from the dual to the primal basis. Thus, the structure of the matrix is preserved and we only compute the non-zero entries: 
\begin{align*}
    \langle e_i,f_i\rangle &= \langle M(\eta_{g-1})(f_i^*) \otimes p^{-\frac{g-1}{2}}, M(\eta_{g-1})(e_i^*) \otimes -p^{-\frac{g-1}{2}} \rangle \\
    &= -p^{-(g-1)} \langle M(\eta_{g-1})(f_i^*), M(\eta_{g-1})(e_i^*)\rangle =-p^{-(g-1)} \langle f_i^*, e_i^*\rangle = p^{-\frac{g-1}{2}}.
\end{align*}
\noindent By antisymmetry, this determines the full matrix representation.
\vspace{4pt}

In the even genus case, the same argument applies blockwise. For odd $i$, we obtain after tensoring: 
\begin{alignat*}{4}
    e_i &= M(\eta_{g-1})(e_{i+1}^*) \otimes p^{-\frac{g}{2}}, &\quad \quad e_{i+1} &= M(\eta_{g-1})(e_i^*) \otimes -p^{-\frac{g}{2}}, \\
    f_i &= M(\eta_{g-1})(f_{i+1}^*) \otimes p^{-\frac{g}{2}+1}, &\quad \quad f_{i+1} &= M(\eta_{g-1})(f_i^*) \otimes -p^{-\frac{g}{2}+1}.
\end{alignat*}
Again, the structure of the matrix is preserved, and the non-zero values are
\begin{align*}
    \langle e_i, e_{i+1}\rangle = -p^{-g}\langle e_{i+1}^*, e_i^*\rangle = p^{-\frac{g}{2}}, \quad
    \langle f_i, f_{i+1}\rangle = -p^{-g+2}\langle f_{i+1}^*, f_i^*\rangle = p^{-\frac{g}{2}+1}.
\end{align*}
With this, antisymmetry fixes the full matrix representation.
\end{proof}

\subsection{Descent criterion}

Having established a normal form for the top level of the PFTQ, we now turn to descending the chain. In particular, we aim to replace the abstract descent condition from Definition \ref{df:PFTQDieudonne}~(iii) by an easily computable descent criterion:

\begin{lem}\label{lem:descentcriterion}
Let $M_i = \langle w, FM_{i+1}, VM_{i+1}\rangle$ with $w \in M_{i+1}$. Assume that the descent condition of a PFTQ is satisfied for all $l > i$, that is 
\begin{align*}
    F^{l-j}V^jM_l \subseteq M_l^t \quad \forall \; 0 \leq j \leq 
    \left\lfloor\smash{\tfrac{l}{2}} \right\rfloor.
\end{align*}
Then, the following criterion assures that the descent condition also holds one step down at level $i$:
{\small
\begin{align*}
    \langle w, F^{i-j}V^jw\rangle \in W(k) \quad \forall \, 0 \leq j \leq \left\lfloor\tfrac{i}{2} \right\rfloor \!\implies\! F^{i-j}V^jM_i \subseteq M_i^t \quad \forall \, 0 \leq j \leq \left\lfloor\tfrac{i}{2} \right\rfloor.
\end{align*}}
\end{lem}

\begin{proof} Firstly, we extend the index range of the descent condition without introducing any new assumptions:

\begin{claim*}
    Let $0 \leq j \leq \left\lfloor \frac{i}{2}\right \rfloor$ be arbitrary. Then, the following inclusion holds:
    \begin{align*}
        F^{(i+1)-(j+1)}V^{j+1}M_{i+1} \subseteq M_{i+1}^t.
    \end{align*}
\end{claim*}
{\renewcommand{\qedsymbol}{$\blacksquare$}
\begin{proof}
If $j <\frac{i}{2}$, then $j+1 \leq \left\lfloor \frac{i+1}{2}\right \rfloor$. In this case, the inclusion follows directly by applying the descent condition at level $l = i+1$ for $j+1$.
\vspace{4pt}

Now, suppose that $j = \frac{i}{2}$. In this case, we cannot directly apply the descent condition as $j+1 \nleq \left\lfloor \frac{i+1}{2}\right \rfloor$. Thus, we have to reduce the running index. For this, let $m_{i+1} \in M_{i+1}$. To establish the claimed inclusion, we have to prove that $$F^{\frac{i}{2}}V^{\frac{i}{2}+1}m_{i+1} \in M_{i+1}^t.$$
By Theorem \ref{thm:QuasipolarisationFull} (ii), this is equivalent to $\langle F^{\frac{i}{2}}V^{\frac{i}{2}+1}m_{i+1},n_{i+1}\rangle\in W(k)$ for all $n_{i+1} \in M_{i+1}$. We can rewrite this term as follows:
\begin{align*}
    \langle n_{i+1}, F^{\frac{i}{2}}V^{\frac{i}{2}+1}m_{i+1}\rangle &=  p^{\frac{i}{2}}\langle n_{i+1}, Vm_{i+1}\rangle = p^{\frac{i}{2}}\sigma^{-1}(\langle Fn_{i+1},m_{i+1}\rangle) \\
    &= \sigma^{-1}(\langle p^{\frac{i}{2}}Fn_{i+1}, m_{i+1}\rangle) =\sigma^{-1}(\langle F^{\frac{i}{2}+1}V^{\frac{i}{2}}n_{i+1},m_{i+1}\rangle).
\end{align*}

\noindent As $\sigma^{-1}(\langle F^{\frac{i}{2}+1}V^{\frac{i}{2}}n_{i+1},m_{i+1}\rangle) \in W(k)$ if and only if $\langle F^{\frac{i}{2}+1}V^{\frac{i}{2}}n_{i+1},m_{i+1}\rangle \in W(k)$, we have reduced the problem to showing that $$F^{\frac{i}{2}+1}V^{\frac{i}{2}}n_{i+1} \in M_{i+1}^t.$$ 

\noindent This holds as $n_{i+1} \in M_{i+1}$ and $F^{\frac{i}{2}+1}V^{\frac{i}{2}}M_{i+1} = F^{i+1-\frac{i}{2}}V^{\frac{i}{2}}M_{i+1} \subseteq M_{i+1}^t$, where we applied the descent condition at level $l = i+1$ for $j = \frac{i}{2}$.
\end{proof}
}

With this claim, we now prove the descent criterion. We define $l = i-2j$ and let $m_i \in M_i$. By assumption, we can write
$$m_i = \lambda_1F(m_{i+1}) + \lambda_2V(\hat{m}_{i+1}) + \lambda_3w,$$
for some $\lambda_k \in W(k)$ and $m_{i+1}, \hat{m}_{i+1} \in M_{i+1}$. With this, we decompose
{\small
\begin{align*}
    F^{i-j}V^jm_i &= F^l(p^jm_i) = p^jF^l(\lambda_1F(m_{i+1}) + \lambda_2V(\hat{m}_{i+1}) + \lambda_3w) \\
    &=\sigma^l(\lambda_1)F^{i+1-j}V^jm_{i+1} + \sigma^l(\lambda_2)F^{i-j}V^{j+1}\hat{m}_{i+1}+\sigma^l(\lambda_3)F^{i-j}V^jw.
\end{align*}}
To conclude that $F^{i-j}V^jm_i \in M_i^t$, it suffices to verify it for each summand:
\begin{itemize}
    \item[(1)]\underline{$F^{i+1-j}V^j(m_{i+1})$}: Here, it holds that $j \leq \left\lfloor \frac{i}{2}\right\rfloor$ and thus $j \leq \left\lfloor \frac{i+1}{2}\right\rfloor$. Applying the descent condition at level $l=i+1$ for $j$, we obtain that $F^{i+1-j}V^j(m_{i+1}) \in M_{i+1}^t \subseteq M_i^t.$
    \item[(2)] \underline{$F^{i-j}V^{j+1}(\hat{m}_{i+1})$}: For this, we use the above claim, which implies that $F^{i-j}V^{j+1}(\hat{m}_{i+1}) = F^{(i+1)-(j+1)}V^{j+1}(\hat{m}_{i+1}) \in M_{i+1}^t \subseteq M_i^t.$ 
    \item[(3)] \underline{$F^{i-j}V^j(w)$}: It is equivalent to show that $\langle n_i, F^{i-j}V^j(w) \rangle \in W(k)$ for all $n_i \in M_i$. Similar to the earlier decomposition, we write $$n_i = \mu_1F(n_{i+1}) + \mu_2V(\hat{n}_{i+1}) + \mu_3w$$ with $\mu_k \in W(k)$ and $n_{i+1}, \hat{n}_{i+1}\in M_{i+1}$. In the quasi-polarisation, this yields three terms for which we verify that they are in $W(k)$.

    The first one is $\langle F(n_{i+1}), F^{i-j}V^j(w)\rangle = \sigma(\langle n_{i+1}, F^{i-j}V^{j+1}(w)\rangle).$ It follows from the above claim that this term lies in $W(k)$.

    The second term $\langle V(\hat{n}_{i+1}), F^{i-j}V^j(w)\rangle = \sigma^{-1}(\langle \hat{n}_{i+1}, F^{i+1-j}V^j(w)\rangle)$ follows similarly as the descent condition implies that it lies in $W(k)$. 

    Lastly, all that is left to prove is that $\langle w, F^{i-j}V^j(w)\rangle \in W(k)$. However, this is exactly the descent criterion we have assumed. \qedhere
\end{itemize}
\end{proof}

\begin{rem}\label{rem:generaliseddescent}
A similar descent criterion can be established if $M_i$ is generated by more elements, say $M_i = \langle v, w, FM_{i+1}, VM_{i+1}\rangle$. In this case, all possible pairs of $v$ and $w$ have to be considered, so that we have to check the following four pairings in the index range $0 \leq j \leq \left\lfloor\tfrac{i}{2}\right\rfloor$ to ensure descent:
\begin{align*}
    \langle v, F^{i-j}V^jv \rangle, \langle w, F^{i-j}V^jv\rangle, \langle v, F^{i-j}V^jw\rangle, \langle w, F^{i-j}V^jw\rangle \in W(k).
\end{align*}
\end{rem}

\subsection{Fiber of the first truncation morphism} Having established these structure theorems, we compute the fibers of the truncation morphisms corresponding to the first step of an arbitrary PFTQ which are higher-dimensional generalisations of the Fermat curve:

\begin{thm}\label{thm:firststepPFTQ}
Let $M_{g-1} \supseteq M_{g-2}$ be the first step of a PFTQ. 
\begin{itemize}
    \item[(i)] For odd $g$, the fiber of the truncation morphism is given by the following equations in $\mathbb{P}^{g-1}$:
    $$\sum_{i=1}^g \alpha_i^{p^{g-2j}+1} = 0 \quad \forall \, j \in \{1, \ldots, \left\lfloor \tfrac{g}{2} \right\rfloor\}.$$
    \item[(ii)] For even $g$, the fiber of the truncation morphism is given by the following equations in $\mathbb{P}^{g-1}$:
    $$\sum_{\substack{i=1 \\ i \text{ odd}}}^g \left(\alpha_i\alpha_{i+1}^{p^{g-2j}}-\alpha_{i+1}\alpha_i^{p^{g-2j}}\right) =0 \quad \forall \, j \in \{1, \ldots, \tfrac{g}{2}\}.$$
\end{itemize}
\end{thm}

\begin{proof}
For both cases, the proof consists of the following two steps:
\begin{itemize}
    \item[(1)] Compute $M_{g-2}$ using the short exact sequence $$0 \to \frac{M_{g-2}}{(F,V)M_{g-1}} \to \frac{M_{g-1}}{(F,V)M_{g-1}} \to \frac{M_{g-1}}{M_{g-2}} \to 0.$$
    \item[(2)] Find the fibers of the truncation morphism via the descent criterion.
\end{itemize}

\noindent In the odd genus case, Theorem \ref{thm:Structure}~(i) gives the following basis of $M_{g-1}$:
\begin{align*}
    M_{g-1} &= \langle e_1, e_2, \ldots, e_g, f_1, f_2, \ldots, f_g\rangle, \\
    (F,V)M_{g-1} &= \langle \pm f_1, \ldots, \pm f_g, \mp pe_1, \ldots, \mp pe_g \rangle = \langle f_1, \ldots, f_g, pe_1, \ldots, pe_g\rangle.
\end{align*}
Therefore, the middle term of the sequence is explicitly given by $$\frac{M_{g-1}}{(F,V)M_{g-1}}= \frac{\langle e_1, \ldots, e_g, f_1, \ldots, f_g\rangle}{\langle pe_1, \ldots, pe_g, f_1, \ldots, f_g\rangle } \cong \langle [e_1], \ldots, [e_g]\rangle \cong \left(\frac{W(k)}{(p)}\right)^g \cong k^g,$$  where $[e_i]$ denotes the class of $e_i$ in the quotient.
To construct the short exact sequence, the third isomorphism theorem gives
$$\frac{M_{g-1}/(F,V)M_{g-1}}{M_{g-2}/(F,V)M_{g-1}}\cong \frac{M_{g-1}}{M_{g-2}},$$
which yields the wanted short exact sequence. Clearly, it is a sequence of $k$-vector spaces, so the dimension of the first term is given by $$\text{dim}_k\left(\tfrac{M_{g-2}}{(F,V)M_{g-1}}\right) = \text{dim}_k\left(\tfrac{M_{g-1}}{(F,V)M_{g-1}}\right) - \text{dim}_k\left(\tfrac{M_{g-1}}{M_{g-2}}\right) = g - (g-1) = 1,$$ where $\text{dim}_k\left(M_{g-1}/M_{g-2}\right) = g-1$ by part (ii) of the definition of a PFTQ.

Thus, $\frac{M_{g-2}}{(F,V)M_{g-1}} = \langle \alpha_1[e_1]+ \ldots + \alpha_g[e_g]\rangle,$ with at least one non-zero $\alpha_i \in k$, i.e.\ $[\alpha_1\mathpunct{:} \ldots\mathpunct{:} \alpha_g]\in \mathbb{P}_k^{g-1}$. By defining $v = \underline{\alpha_1}e_1 + \ldots + \underline{\alpha_g}e_g$, we conclude that
\begin{align*}
    M_{g-2} = \langle\underline{\alpha_1} e_1 + \ldots + \underline{\alpha_g}e_g, f_1, \ldots, f_g, pe_1, \ldots, pe_g\rangle = \langle v, FM_{g-1}, VM_{g-1}\rangle.
\end{align*}

\vspace{4pt}

With this, we are in the setup to apply the descent criterion. For the first step of a PFTQ, we apply it at level $i = g-2$, so we check that
$$\langle v, F^{(g-2)-j}V^jv\rangle \in W(k) \quad \text{for } j = 0, \ldots, \left\lfloor \tfrac{g-2}{2}\right\rfloor.$$
Introducing the convenient index $l = i-2j = g-2(j+1)$, we calculate 
$$F^{(g-2)-j}V^jv = F^l(p^jv) = \sigma^l(\underline{\alpha_1})p^j(-p)^{\frac{l-1}{2}}f_1 + \ldots + \sigma^l(\underline{\alpha_g})p^j(-p)^{\frac{l-1}{2}}f_g.$$
Substituting into the quasi-polarisation and using Corollary \ref{cor:Structure}~(i) gives
{\small
\begin{align*}
    \langle v,F^{(g-2)-j}V^jv\rangle \!&= \underline{\alpha_1}\sigma^l(\underline{\alpha_1})p^j(-p)^{\frac{l-1}{2}}\langle e_1, f_1\rangle + \! \ldots \!+ \underline{\alpha_g}\sigma^l(\underline{\alpha_g})p^j(-p)^{\frac{l-1}{2}}\langle e_g, f_g\rangle \\
    \!&=(-1)^{\frac{l-1}{2}}\left(\underline{\alpha_1}\sigma^l(\underline{\alpha_1})p^{j+\frac{l-1}{2}-\frac{g-1}{2}}\! + \! \ldots \!+ \underline{\alpha_g}\sigma^l(\underline{\alpha_g})p^{j+\frac{l-1}{2}-\frac{g-1}{2}}\!\right).
\end{align*}}
As the exponent of $p$ is $j+\frac{l-1}{2} - \frac{g-1}{2} = -1$, this term simplifies to 
$$\langle v,F^{(g-2)-j}V^jv\rangle = p^{-1}(-1)^{\frac{l-1}{2}}\left(\underline{\alpha_1}\sigma^l(\underline{\alpha_1})+\ldots + \underline{\alpha_g}\sigma^l(\underline{\alpha_g})\right).$$
Obviously, this term lies in $W(k)$ if and only if the coefficient of $p^{-1}$ vanishes: 
$$\underline{\alpha_1}\sigma^l(\underline{\alpha_1})+\ldots + \underline{\alpha_g}\sigma^l(\underline{\alpha_g}) = 0 \iff \alpha_1^{p^{g-2(j+1)}+1} + \ldots + \alpha_g^{p^{g-2(j+1)}+1} = 0,$$
for $j= 0, \ldots, \lfloor \frac{g}{2}\rfloor-1$. Shifting the index gives the claimed form of the fiber.
\vspace{0pt}

In the even genus case, we obtain by the same calculations that $$M_{g-2} = \langle \underline{\alpha_1} e_1 + \ldots + \underline{\alpha_g}e_g, f_1, \ldots, f_g, pe_1, \ldots, pe_g\rangle = \langle v, FM_{g-1}, VM_{g-1}\rangle.$$

\noindent Using the descent criterion to determine the fiber, we calculate
{\small
\begin{align*}
    \langle v, F^{(g-2)-j}V^jv\rangle \! = p^{-1}\!\left(\underline{\alpha_1}\sigma^l(\underline{\alpha_2}) -\underline{\alpha_2}\sigma^l(\underline{\alpha_1})\pm \! \ldots \!+ \underline{\alpha_{g-1}}\sigma^l(\underline{\alpha_g}) -\underline{\alpha_g}\sigma^l(\underline{\alpha_{g-1}})\right).
\end{align*}}
\noindent This is in $W(k)$ if the coefficient of $p^{-1}$ vanishes, giving the condition 
$$\alpha_1\alpha_2^{p^{g-2(j+1)}}-\alpha_2\alpha_1^{p^{g-2(j+1)}} \pm \ldots + \alpha_{g-1}\alpha_g^{p^{g-2(j+1)}}-\alpha_g\alpha_{g-1}^{p^{g-2(j+1)}} = 0,$$
for $j = 0, \ldots, \left\lfloor \frac{g-2}{2}\right\rfloor = \frac{g}{2}-1$. Shifting $j$ gives the claimed fiber.
\end{proof}

\subsection{Fiber of the last truncation morphism} We can similarly investigate the last step of a PFTQ, which exhibits a simpler structure:

\begin{thm}\label{thm:laststepPFTQ}
Let $M_1 \supseteq M_0$ be the last step of a PFTQ. Then, independent of the genus $g$, the fiber of the corresponding truncation morphism is a $\mathbb{P}^1$-bundle over $\mathcal{V}_1$.
\end{thm}

\pagebreak
\begin{proof} The proof is a slightly adapted version of the previous one:
\begin{itemize}
    \item[(1)] Compute $M_0$ using the short exact sequence $$0 \to \frac{M_0}{\text{im}(M(\eta_1))} \to \frac{M_1}{\text{im}(M(\eta_1))} \to \frac{M_1}{M_0} \to 0.$$
    \item[(2)] Verify descent to find the fiber of the truncation morphism.
\end{itemize}

For (1), the sequence follows from the third isomorphism theorem as $$\frac{M_1/\text{im}(M(\eta_1))}{M_0/\text{im}(M(\eta_1))} \cong \frac{M_1}{M_0}.$$
We now calculate the dimensions of the terms in this sequence. Whilst the dimension of the last term is fixed by the definition of a PFTQ, we use Dieudonné theory for the middle term: Under this antiequivalence of categories, $\frac{M_1}{\text{im}(M(\eta_1))} = \text{coker}(M(\eta_1))$ corresponds to $\text{ker}(\eta_1)$. Thus, $$\text{dim}_k\left(\tfrac{M_1}{\text{im}(M(\eta_1))}\right) = \text{log}_p(\text{rk}_k(\text{ker}(\eta_1))) = \text{log}_p(\text{deg}(\eta_1)).$$
By construction of PFTQs, it holds that $\eta_1 = \rho_1^{\vee} \circ \eta_0 \circ \rho_1$. As the degree of isogenies is multiplicative, we determine the degree of $\eta_1$ by considering each map individually.
Noting that dual maps have the same degree as the original, it follows that $\text{deg}(\rho_1^{\vee}) = \text{deg}(\rho_1)$. By definition, $\rho_1$ has an $\alpha$-group of rank $1$ as its kernel, implying $\text{deg}(\rho_1) = p$. Lastly, $\eta_0$ is a principal polarisation and thus of degree $1$. Combining these observations, we get $$\text{deg}(\eta_1) = \text{deg}(\rho_1^{\vee}) \cdot \text{deg}(\eta_0)\cdot \text{deg}(\rho_1)=  \text{deg}(\rho_1)^2 = p^2.$$ 
Therefore, $\text{dim}_k\left(M_1/\text{im}(M(\eta_1))\right) = 2,$ and we fix a basis of it as $[v_1], [v_2]$.
The short exact sequence implies $$\text{dim}_k\left(\tfrac{M_0}{\text{im}(M(\eta_1))}\right) = \text{dim}_k\left(\tfrac{M_1}{\text{im}(M(\eta_1))}\right)-\text{dim}_k\left(\tfrac{M_1}{M_0}\right) = 2-1 = 1.$$
It follows that $\tfrac{M_0}{\text{im}(M(\eta_1))} = \langle \alpha_1[v_1]+\alpha_2[v_2]\rangle,$ where $\alpha_i \in k$ are not both zero, i.e.\ $[\alpha_1\mathpunct{:}\alpha_2]\in \mathbb{P}_k^1$. By defining $v = \underline{\alpha_1}v_1 + \underline{\alpha_2}v_2,$ we conclude that $$M_0 = \langle v, \text{im}(M(\eta_1))\rangle.$$

For (2), we claim that there are no restrictions on $v$ and thus the fiber is the full $\mathbb{P}^1$-bundle. For this, we have to verify that the descent condition $F^{i-j}V^jM_0\subseteq M_0^t$ is satisfied. Since we are at the last step of the PFTQ, $i=0$ and $j=0$. Thus, we only have to check if $M_0 \subseteq M_0^t$. 

As $\text{im}(M(\eta_1)) = M_1^t \subseteq M_0^t$ by definition of a PFTQ, it remains to verify that $v \in M_0^t$. By Theorem \ref{thm:QuasipolarisationFull} (ii), this is equivalent to showing that $\langle v,m\rangle \in W(k)$ for all $m \in M_0$. To this end, we write $m = \lambda_0v+\lambda_1w$ for some $\lambda_i \in W(k)$ and $w \in \text{im}(M(\eta_1)) = M_1^t$. Using bilinearity, we see that
\begin{align*}
    \langle v,m\rangle = \langle v, \lambda_0v + \lambda_1w\rangle = \lambda_0 \langle v,v \rangle + \lambda_1\langle v,w\rangle.
\end{align*}
The first summand vanishes as the quasi-polarisation is alternating, and the second lies in $W(k)$ as $w \in M_1^t \subseteq M_0^t$, showing that the descent condition holds for any $[\alpha_1\mathpunct{:}\alpha_2]\in \mathbb{P}^1_k$.
\end{proof}

\newpage
\section{Explicit calculations in low dimensions}\label{sec:calculations}

In this section, we explicitly compute PFTQs in low dimensions, focusing primarily on the resulting fibers of the truncation morphisms. Partial results in this direction can be found in the literature, see \parencite{Harashita, Pieper, Dina}; however, these sources do not treat this problem systematically and instead focus on isolated steps of the PFTQs in selected dimensions.

\subsection{Dimension \texorpdfstring{$g=2$}{g=2}}\label{sec:g2}

We begin by describing polarised flag type quotients of abelian surfaces. These consist of the data given in the following commutative diagram, the translation of which into the language of Dieudonné theory is given on the right-hand side:

\begin{equation*}
\begin{tikzcd}[column sep = large, row sep = 2.3em]
Y_1 \arrow[r, "\rho_1"] \arrow[d, "\eta_1"'] & Y_0 \arrow[d, "\eta_0"]               &   & M_1                                                         & M_0 \arrow[l, "M(\rho_1)"']   \\
Y_1^{\vee}                                   & Y_0^{\vee} \arrow[l, "\rho_1^{\vee}"] &   & M_1^t \arrow[r, "M(\rho_1^{\vee})"'] \arrow[u, "M(\eta_1)"] & M_0^t \arrow[u, "M(\eta_0)"']
\end{tikzcd}
\end{equation*}

\noindent By Theorem \ref{thm:Structure} and its corollary, we obtain a full description of $M_1$ and its quasi-polarisation: $M_1$ has a skeletal basis $e_1, f_1, e_2, f_2$ and the quasi-polarisation on $M_1$ is given by
\begin{align*}
\begin{pNiceMatrix}[c, margin=2pt]
    \Block{2-2}<\Large>{0} 
    & & p^{-1} & 0 \\
    & & 0 & 1 \\
    -p^{-1} & 0 & \Block{2-2}<\Large>{0}&  \\
    0 & -1 &  &  \\
    \CodeAfter
    \begin{tikzpicture}
        \draw[line width = 0.2pt, dashed] (3-|1) -- (3-|5) ;
        \draw[line width = 0.2pt, dashed] (1-|3) |- (5-|3) ;
    \end{tikzpicture}
\end{pNiceMatrix}.
\end{align*}

To calculate $M_0$, we descend down the PFTQ, recalling that the fiber of the truncation morphism corresponding to the last step is a $\mathbb{P}^1$-bundle. This is consistent with the structure theorem for the first step of a PFTQ: For $g=2$, it states that the fiber is cut out of $\mathbb{P}^1$ by an empty set of equations. 

By the same construction as in the proof of Theorem \ref{thm:Structure}, we obtain 
\begin{align*}\customanchor{eq:g2_0}{}
M_0 = \langle \underline{\alpha_1} e_1 + \underline{\alpha_2} e_2, f_1, pe_1, f_2, pe_2\rangle \quad \text{for some }[\alpha_1\mathpunct{:}\alpha_2]\in \mathbb{P}^1.
\end{align*}

\noindent This parameter also allows us to calculate the $a$-number of $Y_0$:

\begin{pro}\label{pro:anumberY0}
The $a$-number of the abelian variety $Y_0$ is given by
\begin{align*}
    a(Y_0) = \begin{cases}
        2 & \text{if } [\alpha_1\mathpunct{:}\alpha_2]\in\mathbb{P}^1(\mathbb{F}_{p^2}), \\
        1 & \text{if } [\alpha_1\mathpunct{:}\alpha_2]\notin\mathbb{P}^1(\mathbb{F}_{p^2}).
    \end{cases}
\end{align*}
\end{pro}

\begin{proof}
Recall that $a(Y_0) = a(M_0) = \text{dim}_k(M_0/(F,V)M_0)$. For the following computations, we may assume without loss of generality that $\alpha_1 = 1$. In this case, we can find a basis of $M_0$ by removing $pe_1$ from the generators as $pe_1 = p (e_1+\underline{\alpha_2}e_2) - \underline{\alpha_2}(pe_2)$. Therefore, to obtain the $a$-number, we calculate the dimension of $$\frac{M_0}{(F,V)M_0} = \frac{\langle e_1 + \underline{\alpha_2} e_2, f_1, f_2, pe_2\rangle}{\langle f_1 + \sigma(\underline{\alpha_2})f_2, f_1 + \sigma^{-1}(\underline{\alpha_2})f_2, pe_1, pe_2, pf_2\rangle}.$$

Calculating the dimension is now equivalent to finding the rank of the matrix encoding $(F,V)M_0$ inside $M_0$, which can algorithmically be solved by computing the Smith normal form. In later proofs, we will use this algorithmic approach; however, it is instructive to see the hands-on arguments at least once: If $[\alpha_1\mathpunct{:}\alpha_2]\in \mathbb{P}^1(\mathbb{F}_{p^2})$, then the first two generators of the denominator coincide. Therefore, we can compute that
{\small
\begin{align*}
    \frac{M_0}{(F,V)M_0} &= \frac{\langle e_1 + \underline{\alpha_2} e_2, f_1, f_2, pe_2\rangle}{\langle f_1 + \sigma(\underline{\alpha_2})f_2, pe_1, pe_2, pf_2\rangle} = \frac{\langle e_1 + \underline{\alpha_2} e_2, f_1 + \sigma(\underline{\alpha_2})f_2, f_2, pe_2\rangle}{\langle f_1 + \sigma(\underline{\alpha_2})f_2, pe_1, pe_2, pf_2\rangle} \\
    &\cong  \langle [e_1 + \underline{\alpha_2}e_2], [f_2]\rangle \cong \left(\frac{W(k)}{(p)}\right)^2 \cong k^2.
\end{align*}}
If $[\alpha_1\mathpunct{:}\alpha_2]\notin \mathbb{P}^1(\mathbb{F}_{p^2})$, a similar calculation shows the following:
{\small
\begin{align*}
    \frac{M_0}{(F,V)M_0} &= \frac{\langle e_1 + \underline{\alpha_2} e_2, f_1, f_2, pe_2\rangle}{\langle f_1 + \sigma(\underline{\alpha_2})f_2, f_1 + \sigma^{-1}(\underline{\alpha_2})f_2,  pe_1, pe_2, pf_2\rangle} \\
    &= \frac{\langle e_1 + \underline{\alpha_2} e_2, f_1 + \sigma(\underline{\alpha_2})f_2, f_1 + \sigma^{-1}(\underline{\alpha_2})f_2, pe_2\rangle}{\langle f_1 + \sigma(\underline{\alpha_2})f_2, f_1 + \sigma^{-1}(\underline{\alpha_2})f_2,  pe_1, pe_2, pf_2\rangle} \cong \langle [e_1 + \underline{\alpha_2}e_2]\rangle \cong k.
\end{align*}}
Thus, we obtain the claimed dimensions and corresponding $a$-numbers.
\end{proof}

\begin{rem}
It is a general phenomenon that the generic choice of parameters leads to supergeneral abelian varieties. Therefore, we often call such a choice of parameters \emph{supergeneral}.
\end{rem}

\begin{rem}
Comparing the description of $2$-dimensional PFTQs with the data of the Moret-Bailly family \parencite[II. Exemples]{MoretBailly}, we see that the Moret-Bailly parameter and $[\alpha_1\mathpunct{:}\alpha_2]$ carry the same information. Thus, PFTQs may be viewed as generalisations of the Moret-Bailly construction to higher dimensions.
\end{rem}

\subsection{Dimension \texorpdfstring{$g=3$}{g=3}}\label{sec:g3}

The data of $3$-dimensional PFTQs can be represented by the following commutative diagram, the translation of which into the language of Dieudonné theory is given on the right-hand side:
{\small
\begin{equation*}
\begin{tikzcd}[column sep = large, row sep = large]
Y_2 \arrow[r, "\rho_2"] \arrow[d, "\eta_2"'] & Y_1 \arrow[r, "\rho_1"] \arrow[d, "\eta_1"] & Y_0 \arrow[d, "\eta_0"]               & & M_2                                                         & M_1 \arrow[l, "M(\rho_2)"']                                  & M_0 \arrow[l, "M(\rho_1)"']   \\
Y_2^{\vee}                                   & Y_1^{\vee} \arrow[l, "\rho_2^{\vee}"]       & Y_0^{\vee} \arrow[l, "\rho_1^{\vee}"]  &  & M_2^t \arrow[u, "M(\eta_2)"] \arrow[r, "M(\rho_2^{\vee})"'] & M_1^t \arrow[u, "M(\eta_1)"'] \arrow[r, "M(\rho_1^{\vee})"'] & M_0^t \arrow[u, "M(\eta_0)"']
\end{tikzcd}
\end{equation*}}

\noindent As before, the top of the PFTQ is fully described by Theorem \ref{thm:Structure} and its corollary:
In particular, $M_2$ has a non-skeletal basis $e_1, f_1, e_2, f_2, e_3, f_3$ subject to the Frobenius-Verschiebung relations
\begin{align*}
    Fe_i = f_i, \quad Ve_i = -f_i, \quad Ff_i = -pe_i, \quad Vf_i = pe_i.
\end{align*}
With respect to this basis, the quasi-polarisation is given by the matrix
\begin{align*}
\begin{pNiceMatrix}[c, margin=2pt]
    0 & p^{-1} & & & & \\
    -p^{-1} & 0 & &  & & \\
    & &  0 & p^{-1} & & \\
    & & -p^{-1} & 0 & & \\
    & & & &  0 & p^{-1} \\
    & & & & -p^{-1} & 0 \\
    \CodeAfter
    \begin{tikzpicture}
        \draw (3-|1) -- (3-|5) ;
        \draw (5-|3) -- (5-|7) ;
        \draw (1-|3) |- (5-|3) ;
        \draw (3-|5) |- (7-|5) ;
    \end{tikzpicture}
\end{pNiceMatrix}.
\end{align*}

\vspace{10pt}

To determine $M_1$, we descend one level in the PFTQ. As this is the first step, we can apply the black-box Theorem \ref{thm:firststepPFTQ}. Following the same construction as in the proof of that statement, we obtain
$$M_1 = \langle \underline{\alpha_1}e_1 + \underline{\alpha_2}e_2 + \underline{\alpha_3}e_3, FM_2, VM_2 \rangle,$$
and the descent criterion imposes that $[\alpha_1\mathpunct{:} \alpha_2\mathpunct{:} \alpha_3] \in \mathbb{P}^2$ satisfies the equation 
\begin{align}\customanchor{eq:g3_1}{}
    \boxed{\alpha_1^{p+1} + \alpha_2^{p+1} + \alpha_3^{p+1} = 0.}
\end{align}
Note that this is a Fermat curve in $\mathbb{P}^2$. With this, we can calculate the $a$-number of $Y_1$.
\begin{pro}
The $a$-number of the abelian variety $Y_1$ is given by
\begin{align*}
    a(Y_1) = \begin{cases}
        3 & \text{if } [\alpha_1\mathpunct{:}\alpha_2\mathpunct{:}\alpha_3] \in \mathbb{P}^2(\mathbb{F}_{p^2}),\\
        2 & \text{if } [\alpha_1\mathpunct{:}\alpha_2\mathpunct{:}\alpha_3] \notin \mathbb{P}^2(\mathbb{F}_{p^2}).
    \end{cases}
\end{align*}
\end{pro}

\begin{proof}
    Calculate $\dim(M_1/(F,V)M_1)$ analogously to Proposition \ref{pro:anumberY0}.
\end{proof}

At this point, our analysis of $3$-dimensional PFTQs is complete. The last step is governed by Theorem \ref{thm:laststepPFTQ}, which implies that the fiber of the truncation morphism is given by a $\mathbb{P}^1$-bundle over the Fermat curve. This recovers the result also obtained in \parencite{LiOort} within this new framework. Consequently, the structure of $M_0$, and hence of the abelian variety $Y_0$, is determined.

\subsection{Dimension \texorpdfstring{$g=4$}{g=4}}\label{sec:g4}

In this section, we explicitly compute the polarised flag type quotients for $g=4$. Particular care must be taken when selecting parameters, as generic choices only yield supergeneral abelian varieties. Nevertheless, this calculation is instructive as it provides a clean entry point to the more general case involving arbitrary parameter choices. For this reason, we present both approaches in this section.

In both cases, a polarised flag type quotient consists of the same data which is given by the following commutative diagram: 

{\small\begin{equation*}
    \begin{tikzcd}[column sep = 3.9em, row sep = 3.3em]
        Y_3 \arrow[r, "\rho_3"] \arrow[d, "\eta_3"'] & Y_2 \arrow[r, "\rho_2"] \arrow[d, "\eta_2"] & Y_1 \arrow[r, "\rho_1"] \arrow[d, "\eta_1"] & Y_0 \arrow[d, "\eta_0"]               \\
        Y_3^{\vee}                                   & Y_2^{\vee} \arrow[l, "\rho_3^{\vee}"]       & Y_1^{\vee} \arrow[l, "\rho_2^{\vee}"]       & Y_0^{\vee} \arrow[l, "\rho_1^{\vee}"]
    \end{tikzcd}
\end{equation*}}

\noindent The following translation to the language of Dieudonné theory will often be used:

{\small
\begin{equation*}
    \begin{tikzcd}[column sep = 3.9em, row sep = 3.3em]
        M_3                                                         & M_2 \arrow[l, "M(\rho_3)"']                                  & M_1 \arrow[l, "M(\rho_2)"']                                  & M_0 \arrow[l, "M(\rho_1)"']   \\
        M_3^t \arrow[u, "M(\eta_3)"] \arrow[r, "M(\rho_3^{\vee})"'] & M_2^t \arrow[u, "M(\eta_2)"'] \arrow[r, "M(\rho_2^{\vee})"'] & M_1^t \arrow[u, "M(\eta_1)"'] \arrow[r, "M(\rho_1^{\vee})"'] & M_0^t \arrow[u, "M(\eta_0)"']
    \end{tikzcd}
\end{equation*}}

\noindent We begin at the top of the PFTQ, where the structure is fully described by the black-box Theorem \ref{thm:Structure}. In dimension $g=4$, $M_3$ has a skeletal basis $e_1, f_1, e_2, f_2, e_3, f_3, e_4, f_4$, so the Frobenius and Verschiebung relations are given by
\begin{align*}
    Fe_i = f_i, \quad Ve_i = f_i, \quad Ff_i = pe_i, \quad Vf_i = pe_i.
\end{align*}
As shown in Corollary \ref{cor:Structure}, the quasi-polarisation decomposes into two genus-$2$ blocks. With respect to the above skeletal basis, it is thus represented by the following matrix:

{\small
\begin{align*}
\begin{pNiceMatrix}[c, margin=2pt]
    \Block{2-2}<\Large>{0} &    & p^{-2} & 0    & & & & \\
     &                          & 0 & p^{-1}    & & & & \\
    -p^{-2} & 0 & \Block{2-2}<\Large>{0} &      & & & & \\
    0 & -p^{-1} &                        &      & & & & \\
    & & & & \Block{2-2}<\Large>{0}  & & p^{-2} & 0      \\
    & & & &                         & & 0 & p^{-1}      \\
    & & & & -p^{-2} & 0 & \Block{2-2}<\Large>{0} &      \\
    & & & & 0 & -p^{-1} &                        &      \\
    \CodeAfter
    \begin{tikzpicture}
        \draw[line width = 0.2pt, dashed] (3-|1) -- (3-|5) ;
        \draw[line width = 0.2pt, dashed] (1-|3) |- (5-|3) ;
        \draw[line width = 0.4pt] (1-|5) |- (9-|5) ;
        \draw[line width = 0.4pt] (5-|1) -- (5-|9) ;
        \draw[line width = 0.2pt, dashed] (7-|5) -- (7-|9) ;
        \draw[line width = 0.2pt, dashed] (5-|7) |- (9-|7) ;
    \end{tikzpicture}
\end{pNiceMatrix}.
\end{align*}}

We determine $M_2$ by descending one level in the PFTQ. As this is the first step, we can apply the black-box Theorem \ref{thm:firststepPFTQ}. Using this, we obtain
\begin{align*}
    M_2 = \langle \underbrace{\underline{\alpha_1}e_1 + \underline{\alpha_2}e_2 + \underline{\alpha_3}e_3 + \underline{\alpha_4}e_4}_{=:v}, f_1, pe_1, f_2, pe_2, f_3, pe_3, f_4, pe_4\rangle.
\end{align*}

Interpreting the parameters as $[\alpha_1\mathpunct{:}\alpha_2\mathpunct{:}\alpha_3\mathpunct{:}\alpha_4]\in \mathbb{P}^3$ since at least one of them has to be non-zero, the descent criterion imposes the equation 
\begin{align}\customanchor{eq:g4_2}{}
    \boxed{\alpha_1 \alpha_2^{p^2} - \alpha_2\alpha_1^{p^2} + \alpha_3\alpha_4^{p^2} - \alpha_4\alpha_3^{p^2}=0.}
\end{align}

To progress further in the PFTQ, it is essential to understand $M_2$ in greater detail. The most important step is to calculate the $a$-number depending on the parameter: 

\begin{pro}\label{pro:anumberY2}
    The $a$-number of the abelian variety $Y_2$ is given by
    \begin{align*}
    a(Y_2) = \begin{cases}
        4 & \text{if } [\alpha_1 \mathpunct{:}\alpha_2\mathpunct{:}\alpha_3\mathpunct{:}\alpha_4] \in \mathbb{P}^3(\mathbb{F}_{p^2}),\\
        3 & \text{if } [\alpha_1\mathpunct{:}\alpha_2\mathpunct{:}\alpha_3\mathpunct{:}\alpha_4] \notin \mathbb{P}^3(\mathbb{F}_{p^2}).
        \end{cases}
    \end{align*}
\end{pro}

\begin{proof}
    Calculate $\dim(M_2/(F,V)M_2)$ as in Proposition \ref{pro:anumberY0}.
\end{proof}

\subsubsection{The supergeneral parameter choice}

\noindent In this subsection, we make the supergeneral parameter choice and assume that $[\alpha_1\mathpunct{:}\alpha_2\mathpunct{:}\alpha_3\mathpunct{:}\alpha_4]\notin \mathbb{P}^3(\mathbb{F}_{p^2})$.

\vspace{8pt}

\noindent To calculate $M_1$, we descend one more level in the PFTQ. For this, we proceed in two steps:
\begin{itemize}
    \item[(1)] Establish a candidate for $M_1$ by establishing the short exact sequence
    $$0 \to \frac{M_1}{(F,V)M_2} \to \frac{M_2}{(F,V)M_2} \to \frac{M_2}{M_1} \to 0.$$
    \item[(2)] Find the fiber of the truncation morphism via the descent criterion.
\end{itemize}

\noindent For (1), we can use the third isomorphism theorem to establish the short exact sequence as
\begin{align*}
    \frac{M_2/(F,V)M_2}{M_1/(F,V)M_2} \cong \frac{M_2}{M_1}.
\end{align*}

To gain insight into $M_1$, we calculate the dimensions of the terms in the short exact sequence. Whilst the dimension of the last term is fixed to be $2$ by definition of the PFTQ, the middle term is of dimension $3$ as we choose the parameters generically, see Proposition \ref{pro:anumberY2}. Therefore, we deduce $$\text{dim}_k\left(\tfrac{M_1}{(F,V)M_2}\right)= \text{dim}_k\left(\tfrac{M_2}{(F,V)M_2}\right) - \text{dim}_k\left(\tfrac{M_2}{M_1}\right) = 3-2 = 1.$$
Using the explicit representation of $M_2$ given above, it is an easy calculation to see that $\frac{M_2}{(F,V)M_2} \cong \langle [v], [f_2], [f_4]\rangle,$ and thus $$\frac{M_1}{(F,V)M_2}\cong \langle \alpha_5[v]+\alpha_6[f_2]+\alpha_7[f_4]\rangle \quad \text{for some }[\alpha_5\mathpunct{:}\alpha_6\mathpunct{:}\alpha_7]\in \mathbb{P}^2.$$

\noindent This can be lifted to a generating set of $M_1$ as follows: $$M_1 = \langle \underbrace{\underline{\alpha_5}v + \underline{\alpha_6}f_2 + \underline{\alpha_7}f_4}_{=:w}, FM_2, VM_2\rangle.$$

\noindent For (2), we are in the setup of the descent criterion, and we apply it at level $i = 1$ and thus for $j=0$. Therefore, it suffices to verify that $\langle w,Fw\rangle \in W(k)$. 

To this end, we firstly compute that
{\small
\begin{align*}
    Fw &= F(\underline{\alpha_5}v+\underline{\alpha_6}f_2+\underline{\alpha_7}f_4) = \sigma(\underline{\alpha_5})Fv + \sigma(\underline{\alpha_6})Ff_2 + \sigma(\underline{\alpha_7})Ff_4 \\
    &=\sigma(\underline{\alpha_5})(\sigma(\underline{\alpha_1})f_1 + \sigma(\underline{\alpha_2})f_2+ \sigma(\underline{\alpha_3})f_3 + \sigma(\underline{\alpha_4})f_4)+ \sigma(\underline{\alpha_6})pe_2 + \sigma(\underline{\alpha_7})pe_4.
\end{align*}}

\noindent Using the representation of the quasi-polarisation on $M_3$, we compute
{\small
\begin{align*}
    \langle w, Fw\rangle &= \langle \underline{\alpha_5}(\underline{\alpha_1}e_1 + \underline{\alpha_2}e_2 + \underline{\alpha_3}e_3 + \underline{\alpha_4}e_4) + \underline{\alpha_6}f_2 + \underline{\alpha_7}f_4, \\
    &\quad\quad \sigma(\underline{\alpha_5})(\sigma(\underline{\alpha_1})f_1 + \sigma(\underline{\alpha_2})f_2+ \sigma(\underline{\alpha_3})f_3 + \sigma(\underline{\alpha_4})f_4)+ \sigma(\underline{\alpha_6})pe_2 + \sigma(\underline{\alpha_7})pe_4\rangle \\
    &=\underline{\alpha_5}\underline{\alpha_1}\sigma(\underline{\alpha_6})p\langle e_1,e_2\rangle + \underline{\alpha_5}\underline{\alpha_3}\sigma(\underline{\alpha_7})p\langle e_3, e_4\rangle + \underline{\alpha_6}\sigma(\underline{\alpha_5}\underline{\alpha_1})\langle f_2, f_1\rangle + \underline{\alpha_7}\sigma(\underline{\alpha_5}\underline{\alpha_3})\langle f_4, f_3 \rangle \\
    &= p^{-1}(\underline{\alpha_5}\underline{\alpha_1}\sigma(\underline{\alpha_6}) + \underline{\alpha_5}\underline{\alpha_3}\sigma(\underline{\alpha_7}) - \underline{\alpha_6}\sigma(\underline{\alpha_5}\underline{\alpha_1}) - \underline{\alpha_7}\sigma(\underline{\alpha_5}\underline{\alpha_3})) \in W(k).
\end{align*}}

\noindent As we are working with Teichmüller lifts, this condition is equivalent to 
\begin{align}\customanchor{eq:g4_1_gen}{}
\boxed{\alpha_5(\alpha_1\alpha_6^p + \alpha_3\alpha_7^p - \alpha_6 \alpha_5^{p-1}\alpha_1^p - \alpha_7 \alpha_5^{p-1}\alpha_3^p) = 0,}
\end{align}
thus giving us the fiber of the truncation morphism $\mathcal{V}_{1} \to \mathcal{V}_{2}$.

\begin{rem}
    It is clear that this variety consists of two components, and that the component defined by $\alpha_5 = 0$ is a \emph{garbage component} in the sense of \parencite{LiOort}: This means that this component blows down to a proper closed subset of the moduli space, or similarly, that choosing this parameter gives rise to non-rigid PFTQs. One can easily verify this behaviour:
\end{rem}

\begin{pro}\label{pro:garbagecomponent}
    Choosing $\alpha_5 =0$ yields a non-rigid PFTQ, and therefore it defines a garbage component of the fiber.
\end{pro}

\begin{proof}
    Since we are working with a skeletal basis of $M_3$, it holds that $FM_3 = VM_3$. Choosing $\alpha_5 =0$ gives $M_1 = \langle \underline{\alpha_6}f_2 + \underline{\alpha_7}f_4, FM_2, VM_2\rangle = \langle FM_2, VM_2\rangle.$ Therefore, we obtain that $$M_1 = (F,V)M_2 \subseteq (F,V)M_3 = FM_3,$$ where the inclusion holds as $M_2 \subseteq M_3$ by definition. Consequently, we have $$M_0 + FM_3 \subseteq M_1 + FM_3 \subseteq FM_3.$$

    However, the rigidity condition of PFTQs implies that $M_0 + FM_3 = M_2$. Combining both statements leads to $M_2 \subseteq FM_3$, which obviously does not hold as $v \in M_2 \setminus FM_3$.
\end{proof}

With this, we have fully understood the construction of $M_1$. The last step of the PFTQ is described by the black-box Theorem \ref{thm:laststepPFTQ}, which implies that the fiber of the truncation morphism $\mathcal{V}_0 \to \mathcal{V}_1$ is a $\mathbb{P}^1$-bundle. Consequently, the structure of $M_0$ is determined.

\begin{rem}
    The preceding calculation recovers the description of the generic $a=1$ locus obtained by Viehmann \parencite{Viehmann} in the language of PFTQs. In the following subsection, we extend the calculation beyond this generic case by allowing arbitrary parameter choices and thus describing the complete moduli space. 
\end{rem}

\subsubsection{The non-supergeneral parameter choice}
We return to the point before assuming that the parameter is chosen generically and revisit this situation without imposing conditions on $[\alpha_1\mathpunct{:}\alpha_2\mathpunct{:}\alpha_3\mathpunct{:}\alpha_4]$ this time. To compute $M_1$, we again proceed in two steps, but now use a sequence that is adapted to the situation and thus independent of the $a$-number:
\begin{itemize}
    \item[(1)] Find a candidate for $M_1$ by using the short exact sequence
    \begin{align*}
        0 \to \frac{M_1}{FM_2} \to \frac{M_2}{FM_2} \to \frac{M_2}{M_1} \to 0.
    \end{align*}
    \item[(2)] Apply a generalised version of the descent criterion to find the fiber of the truncation morphism.
\end{itemize}

\noindent For (1), the third isomorphism theorem yields the short exact sequence as $$\frac{M_2/FM_2}{M_1/FM_2} \cong \frac{M_2}{M_1}.$$

We compute the dimension of the first term to obtain information on $M_1$. Whilst the dimension of the last term is fixed to be $2$ by definition of the PFTQ, the middle term is more subtle: By Lemma \ref{lem:PFTQLemma} (ii), its dimension is given by the genus $g = 4$ and is independent of the choice of parameter. Therefore, we compute
$$\text{dim}_k\left(\tfrac{M_1}{FM_2}\right) = \text{dim}_k\left(\tfrac{M_2}{FM_2}\right)-\text{dim}_k\left(\tfrac{M_2}{M_1}\right) = 4-2 = 2.$$

\noindent It is an easy calculation to find four basis elements of $M_2/FM_2$, namely $$\frac{M_2}{FM_2}\cong \langle [\underline{\alpha_1}e_1 + \underline{\alpha_2}e_2 + \underline{\alpha_3}e_3 + \underline{\alpha_4}e_4], [f_2], [f_3], [f_4]\rangle.$$ The two-dimensional subspace $M_1/FM_2$ of $M_2/FM_2$ is thus given by two non-zero elements, and without loss of generality they are of the form
\begin{align*}
    \frac{M_1}{FM_2} = \langle \alpha_5[v] + \alpha_6[f_2] + \alpha_7[f_3] + \alpha_8[f_4], \alpha_9[f_2] + \alpha_{10}[f_3] + \alpha_{11}[f_4]\rangle,
\end{align*}
where
$[\alpha_5\mathpunct{:}\alpha_6\mathpunct{:}\alpha_7\mathpunct{:}\alpha_8]\in \mathbb{P}^3$ and $[\alpha_9\mathpunct{:}\alpha_{10}\mathpunct{:}\alpha_{11}]\in\mathbb{P}^2$. This can be lifted to a generating set of $M_1$ as follows: $$M_1 = \langle \underbrace{\underline{\alpha_5}v + \underline{\alpha_6}f_2 + \underline{\alpha_7}f_3 + \underline{\alpha_8}f_4}_{=:w_1}, \underbrace{\underline{\alpha_9}f_2 + \underline{\alpha_{10}}f_3 + \underline{\alpha_{11}}f_4}_{=:w_2},  FM_2\rangle.$$

\noindent For (2), we want to ensure that the polarisation descends to $M_1$. For this, we can use a version of the generalised descent criterion. As we add two Witt vectors $w_1, w_2$ to the generating set, we have to check the four conditions
\begin{align*}
    \langle w_1, Fw_1\rangle, \langle w_1, Fw_2\rangle, \langle w_2, Fw_1\rangle, \langle w_2, Fw_2\rangle \in W(k).
\end{align*}

\noindent Before examining these conditions, we precompute $Fw_1$ and $Fw_2$:
{\small
\begin{align*}
    Fw_1 \!&=\! F(\underline{\alpha_5}v \! + \! \underline{\alpha_6}f_2 \!+\! \underline{\alpha_7}f_3 \!+\!\underline{\alpha_8}f_4) = \sigma(\underline{\alpha_5})Fv \!+\! \sigma(\underline{\alpha_6})pe_2 \!+\! \sigma(\underline{\alpha_7})pe_3 \!+\! \sigma(\underline{\alpha_8})pe_4 \\
    &=\! \sigma(\underline{\alpha_5})(\sigma(\underline{\alpha_1})f_1 \!+\! \sigma(\underline{\alpha_2})f_2 \!+\! \sigma(\underline{\alpha_3})f_3 \!+\! \sigma(\underline{\alpha_4})f_4) \!+\! \sigma(\underline{\alpha_6})pe_2 \!+\! \sigma(\underline{\alpha_7})pe_3 \!+\! \sigma(\underline{\alpha_8})pe_4, \\[4pt]
    Fw_2 \!&=\! F(\underline{\alpha_9}f_2 + \underline{\alpha_{10}}f_3 + \underline{\alpha_{11}}f_4) = \sigma(\underline{\alpha_9})pe_2 + \sigma(\underline{\alpha_{10}})pe_3 + \sigma(\underline{\alpha_{11}})pe_4.
\end{align*}}

\noindent Using the representation of the quasi-polarisation, the first condition gives
{\small
\begin{align*}
    \langle w_1, Fw_1\rangle &=\underline{\alpha_5\alpha_1}\sigma(\underline{\alpha_6})p\langle e_1, e_2 \rangle + \underline{\alpha_5\alpha_3}\sigma(\underline{\alpha_8})p\langle e_3, e_4\rangle +\underline{\alpha_5\alpha_4}\sigma(\underline{\alpha_7})p\langle e_4, e_3\rangle \\ 
    & \quad \quad + \underline{\alpha_6}\sigma(\underline{\alpha_5\alpha_1})\langle f_2, f_1\rangle + \underline{\alpha_7}\sigma(\underline{\alpha_5\alpha_4})\langle f_3, f_4\rangle + \underline{\alpha_8}\sigma(\underline{\alpha_5\alpha_3})\langle f_4, f_3\rangle \\
    &= p^{-1}(\underline{\alpha_5\alpha_1}\sigma(\underline{\alpha_6}) + \underline{\alpha_5\alpha_3}\sigma(\underline{\alpha_8}) - \underline{\alpha_5\alpha_4}\sigma(\underline{\alpha_7}) \\
    &\quad\quad -\underline{\alpha_6}\sigma(\underline{\alpha_5\alpha_1}) + \underline{\alpha_7}\sigma(\underline{\alpha_5\alpha_4})-\underline{\alpha_8}\sigma(\underline{\alpha_5\alpha_3})) \in W(k).
\end{align*}}

\noindent Similarly, the second condition simplifies to
{\small
\begin{align*}
    \langle w_1, Fw_2\rangle &= \underline{\alpha_5\alpha_1}\sigma(\underline{\alpha_9}) p\langle e_1, e_2\rangle + \underline{\alpha_5\alpha_3}\sigma(\underline{\alpha_{11}})p\langle e_3, e_4\rangle + \underline{\alpha_5\alpha_4}\sigma(\underline{\alpha_{10}})p\langle e_4,e_3\rangle \\
    &= p^{-1}(\underline{\alpha_5\alpha_1}\sigma(\underline{\alpha_9}) + \underline{\alpha_5\alpha_3}\sigma(\underline{\alpha_{11}})-\underline{\alpha_5\alpha_4}\sigma(\underline{\alpha_{10}})) \in W(k).
\end{align*}}

\noindent Computing the third condition, we obtain
{\small
\begin{align*}
    \langle w_2, Fw_1\rangle &= \underline{\alpha_9}\sigma(\underline{\alpha_5\alpha_1})\langle f_2, f_1 \rangle + \underline{\alpha_{10}}\sigma(\underline{\alpha_5\alpha_4})\langle f_3, f_4 \rangle + \underline{\alpha_{11}} \sigma(\underline{\alpha_5\alpha_3})\langle f_4, f_3 \rangle \\
    &= p^{-1}(-\underline{\alpha_9}\sigma(\underline{\alpha_5\alpha_1})+\underline{\alpha_{10}}\sigma(\underline{\alpha_5\alpha_4})-\underline{\alpha_{11}}\sigma(\underline{\alpha_5\alpha_3})) \in W(k).
\end{align*}}

\noindent Finally, the last condition is automatically satisfied:
{\small
\begin{align*}
    \langle w_2, Fw_2 \rangle &= \langle \underline{\alpha_9}f_2 \!+\! \underline{\alpha_{10}}f_3 \!+\! \underline{\alpha_{11}}f_4, \sigma(\underline{\alpha_9})pe_2 \!+\! \sigma(\underline{\alpha_{10}})pe_3 \!+\! \sigma(\underline{\alpha_{11}})pe_4\rangle \!=\! 0 \in W(k).
\end{align*}} 

To satisfy the three non-trivial conditions, the coefficients of $p^{-1}$ have to vanish. Using Teichmüller lifts, this is equivalent to the three equations
\begin{equation}
\boxed{\customanchor{eq:g4_1_arb}{}
\begin{gathered}
    \alpha_5(\alpha_1\alpha_6^p + \alpha_3\alpha_8^p-\alpha_4 \alpha_7^p - \alpha_6\alpha_5^{p-1}\alpha_1^p + \alpha_7\alpha_5^{p-1}\alpha_4^p - \alpha_8\alpha_5^{p-1}\alpha_3^p) = 0, \\
    \alpha_5(\alpha_1\alpha_9^p+ \alpha_3 \alpha_{11}^p - \alpha_4\alpha_{10}^p) = 0, \\
    \alpha_5^p(-\alpha_9\alpha_1 ^p + \alpha_{10}\alpha_4^p-\alpha_{11}\alpha_3^p) =0,
\end{gathered}}
\end{equation}
and these give the fiber of the truncation morphism $\mathcal{V}_1 \to \mathcal{V}_2$. Comparing with the supergeneral parameter choice, we make two remarks:

\begin{rem}
We can recover the equation derived for the supergeneral parameter choice from the above equations: Assuming $[\alpha_1\mathpunct{:}\alpha_2\mathpunct{:}\alpha_3\mathpunct{:}\alpha_4] \notin \mathbb{P}^3(\mathbb{F}_{p^2})$, we may choose $w_2 = (F-V)v\neq 0.$ Without loss of generality, we assume that $\alpha_1 = 1$. Then, we see that $$\alpha_9 = \sigma(\alpha_2)-\sigma^{-1}(\alpha_2), \alpha_{10} = \sigma(\alpha_3)-\sigma^{-1}(\alpha_3), \alpha_{11} = \sigma(\alpha_4)-\sigma^{-1}(\alpha_4).$$

\noindent Using these identities, the second equation is automatically satisfied:
\begin{align*}
    \alpha_1\alpha_9^p + \alpha_3\alpha_{11}^p-\alpha_4\alpha_{10}^p &= \alpha_1(\alpha_2^{p^2}-\alpha_2) + \alpha_3(\alpha_4^{p^2}-\alpha_4) -\alpha_4(\alpha_3^{p^2}-\alpha_3) \\
    &= \alpha_1\alpha_2^{p^2}-\alpha_1\alpha_2 + \alpha_3\alpha_4^{p^2}-\alpha_4\alpha_3^{p^2} \\
    &= \alpha_1 \alpha_2^{p^2}-\alpha_1\alpha_2 -\alpha_1\alpha_2^{p^2}+\alpha_1^{p^2}\alpha_2 = 0,
\end{align*}
where we used the identity imposed by the descent criterion on the parameters during descent from $M_3$ to $M_2$, as well as the fact that $\alpha_1 = 1$. A similar computation shows that the third equation is also automatically satisfied. Therefore, we are left with the first equation, which is a globalised version of the fiber equation for the supergeneral parameter choice.
\end{rem}

\begin{rem}
    The same argument as in the supergeneral case shows that the component given by $\alpha_5=0$ leads to non-rigid PFTQs, and thus defines a garbage component of the fiber. 
\end{rem}

The last step of the PFTQ is described by the black-box Theorem \ref{thm:laststepPFTQ} as a $\mathbb{P}^1$-bundle over $\mathcal{V}_1$, consequently determining the structure of $M_0$. 

\subsection{Dimension \texorpdfstring{$g=5$}{g=5}}\label{sec:g5}

In this section, we calculate the polarised flag type quotients in dimension $g=5$ for generic and arbitrary parameter choices. Similar to the previous section, we provide both calculations as the generic parameter choice is of interest in its own right and provides an accessible entry point to the arbitrary case. In both cases, the data of a $5$-dimensional polarised flag type quotient can be represented by the commutative diagram:

{\small
\begin{equation*}
    \begin{tikzcd}[column sep = 3.9em, row sep = 3.3em]
        Y_4 \arrow[r, "\rho_4"] \arrow[d, "\eta_4"'] & Y_3 \arrow[r, "\rho_3"] \arrow[d, "\eta_3"] & Y_2 \arrow[r, "\rho_2"] \arrow[d, "\eta_2"] & Y_1 \arrow[r, "\rho_1"] \arrow[d, "\eta_1"] & Y_0 \arrow[d, "\eta_0"]               \\
        Y_4^{\vee}                                   & Y_3^{\vee} \arrow[l, "\rho_4^{\vee}"]       & Y_2^{\vee} \arrow[l, "\rho_3^{\vee}"]       & Y_1^{\vee} \arrow[l, "\rho_2^{\vee}"]       & Y_0^{\vee} \arrow[l, "\rho_1^{\vee}"]
\end{tikzcd}
\end{equation*}}

\noindent This diagram can be translated into the language of Dieudonné theory: 

{\small
\begin{equation*}
    \begin{tikzcd}[column sep = 3.9em, row sep = 3.3em]
        M_4                                                         & M_3 \arrow[l, "M(\rho_4)"']                                  & M_2 \arrow[l, "M(\rho_3)"']                                  & M_1 \arrow[l, "M(\rho_2)"']                                  & M_0 \arrow[l, "M(\rho_1)"']   \\
        M_4^t \arrow[u, "M(\eta_4)"] \arrow[r, "M(\rho_4^{\vee})"'] & M_3^t \arrow[u, "M(\eta_3)"'] \arrow[r, "M(\rho_3^{\vee})"'] & M_2^t \arrow[u, "M(\eta_2)"'] \arrow[r, "M(\rho_2^{\vee})"'] & M_1^t \arrow[u, "M(\eta_1)"'] \arrow[r, "M(\rho_1^{\vee})"'] & M_0^t \arrow[u, "M(\eta_0)"']
\end{tikzcd}
\end{equation*}}

\vspace{8pt}

\noindent We begin at the top of the PFTQ which is fully described by Theorem \ref{thm:Structure}: In particular, $M_4$ has a non-skeletal basis $e_1, f_1, \ldots, e_5, f_5$ subject to the Frobenius-Verschiebung relations
\begin{align*}
    Fe_i = f_i, \quad Ve_i = -f_i, \quad Ff_i = -pe_i, \quad Vf_i = pe_i.
\end{align*}

\noindent In this dimension, the quasi-polarisation decomposes into five genus-$1$ blocks, see Corollary \ref{cor:Structure}. With respect to the basis, it is represented by the matrix

{\small
\begin{align*}
\begin{pNiceMatrix}[c, margin=2pt]
    0 & p^{-2} & & & & & & & & \\
    -p^{-2} & 0 & & & & & & & & \\
    & & 0 & p^{-2} & & & & & & \\
    & & -p^{-2} & 0 & & & & & & \\
    & & & & 0 & p^{-2} & & & & \\
    & & & & -p^{-2} & 0 & & & & \\
    & & & & & & 0 & p^{-2} & & \\
    & & & & & & -p^{-2} & 0 & & \\
    & & & & & & & & 0 & p^{-2} \\
    & & & & & & & & -p^{-2} & 0 \\
    \CodeAfter
    \begin{tikzpicture}
        \draw (3-|1) -- (3-|5) ;
        \draw (5-|3) -- (5-|7) ;
        \draw (7-|5) -- (7-|9) ;
        \draw (9-|7) -- (9-|11);
        \draw (1-|3) |- (5-|3) ;
        \draw (3-|5) |- (7-|5) ;
        \draw (5-|7) |- (9-|7) ;
        \draw (7-|9) |- (11-|9);
    \end{tikzpicture}
\end{pNiceMatrix}.
\end{align*}}

To determine $M_3$, we descend one level in the PFTQ. As this is the first step, we can apply the black-box Theorem \ref{thm:firststepPFTQ}, which implies that $$M_3 = \langle \underbrace{\underline{\alpha_1}e_1 + \underline{\alpha_2}e_2 + \underline{\alpha_3}e_3 + \underline{\alpha_4}e_4 + \underline{\alpha_5}e_5}_{=:v}, FM_4, VM_4\rangle.$$ 
Interpreting the parameters as $[\alpha_1\mathpunct{:}\alpha_2\mathpunct{:}\alpha_3\mathpunct{:}\alpha_4\mathpunct{:}\alpha_5]\in\mathbb{P}^4$, the descent criterion imposes the equations
\begin{equation}
\boxed{\customanchor{eq:g5_3}{}\begin{gathered}
    \alpha_1^{p^3+1} + \alpha_2^{p^3+1} + \alpha_3^{p^3+1} + \alpha_4^{p^3+1} + \alpha_5^{p^3+1} = 0, \\
    \alpha_1^{p+1} + \alpha_2^{p+1} + \alpha_3^{p+1} + \alpha_4^{p+1} + \alpha_5^{p+1} = 0.
\end{gathered}}
\end{equation}
In order to progress down the PFTQ with the usual two-step approach, we need to calculate the $a$-number of $M_3$ depending on the parameters:

\begin{pro}\label{pro:anumberY3}
    The $a$-number of the abelian variety $Y_3$ is given by
    \begin{align*}
    a(Y_3) = 
        \begin{cases}
            5 & \text{if } [\alpha_1\mathpunct{:}\alpha_2\mathpunct{:}\alpha_3\mathpunct{:}\alpha_4\mathpunct{:}\alpha_5] \in \mathbb{P}^4(\mathbb{F}_{p^2}), \\
            4 & \text{if } [\alpha_1\mathpunct{:}\alpha_2\mathpunct{:}\alpha_3\mathpunct{:}\alpha_4\mathpunct{:}\alpha_5] \notin \mathbb{P}^4(\mathbb{F}_{p^2}).
        \end{cases}
    \end{align*}
\end{pro}

\begin{proof}
    Compute $\dim(M_3/(F,V)M_3)$ analogously to Proposition \ref{pro:anumberY0}.
\end{proof}

\subsubsection{The supergeneral parameter choice}

\noindent Similar to the generic calculations in lower dimensions, we now assume that 
$[\alpha_1\mathpunct{:}\alpha_2\mathpunct{:}\alpha_3\mathpunct{:}\alpha_4\mathpunct{:}\alpha_5] \notin \mathbb{P}^4(\mathbb{F}_{p^2})$. To descend one level down the PFTQ to $M_2$, we proceed in two steps:
\begin{itemize}
    \item[(1)] Establish a candidate for $M_2$ by establishing the short exact sequence
    \begin{align*}
        0 \to \frac{M_2}{(F,V)M_3} \to \frac{M_3}{(F,V)M_3} \to \frac{M_3}{M_2} \to 0.
    \end{align*}
    \item[(2)] Find the fiber of the truncation morphism via the descent criterion.
\end{itemize}

\vspace{4pt}

\noindent For (1), the short exact sequence arises from the third isomorphism theorem: $$\frac{M_3/(F,V)M_3}{M_2/(F,V)M_3} \cong \frac{M_3}{M_2}.$$

\noindent Having established this short exact sequence, the next step is to calculate the dimensions of the terms. We note that the middle term has dimension $4$ by the generic choice of the parameter, and the dimension of the last term is fixed to be $3$ by definition of the PFTQ. Therefore, we find that
\begin{align*}
    \text{dim}_k\left(\tfrac{M_2}{(F,V)M_3}\right) = \text{dim}_k\left(\tfrac{M_3}{(F,V)M_3}\right)-\text{dim}_k\left(\tfrac{M_3}{M_2}\right) = 4-3 = 1.
\end{align*}

\noindent Using the representation of $M_3$ given above, it is easy to verify that $\frac{M_3}{(F,V)M_3}\cong \langle [v], [f_3], [f_4], [f_5]\rangle$, and therefore the one-dimensional subspace is given by $$\frac{M_2}{(F,V)M_3} = \langle \alpha_6[v] + \alpha_7[f_3] + \alpha_8[f_4] + \alpha_9[f_5]\rangle,$$ where $[\alpha_6\mathpunct{:}\alpha_7\mathpunct{:}\alpha_8\mathpunct{:}\alpha_9] \in \mathbb{P}^3.$ This can be lifted to a generating set of $M_2$:
\begin{align*}
    M_2 = \langle \underbrace{\underline{\alpha_6}v + \underline{\alpha_7}f_3 + \underline{\alpha_8}f_4 + \underline{\alpha_9}f_5}_{=:w}, FM_3, VM_3\rangle.
\end{align*}

\noindent For (2), we want to ensure that the polarisation descends to $M_2$ and hence apply the descent criterion at level $i=2$. Thus, we have to check that $$\langle w, F^2w\rangle \in W(k), \quad \langle w, FVw\rangle  \in W(k).$$

\noindent We firstly precompute the image of $w$ under the square of the Frobenius:
{\small
\begin{align*}
    F^2w &= \sigma^2(\underline{\alpha_6})F^2v+ \sigma^2(\underline{\alpha_7})F^2f_3 + \sigma^2(\underline{\alpha_8})F^2f_4 + \sigma^2(\underline{\alpha_9})F^2f_5 \\
    &= -p(\sigma^2(\underline{\alpha_6\alpha_1})e_1 + \ldots +\sigma^2(\underline{\alpha_6\alpha_5})e_5 + \sigma^2(\underline{\alpha_7})f_3 + \sigma^2(\underline{\alpha_8})f_4 + \sigma^2(\underline{\alpha_9})f_5).
\end{align*}}

\noindent Using the representation of the quasi-polarisation, the first condition gives
{\small
\begin{align*}
    \langle w, F^2w\rangle &= \langle \underline{\alpha_6}(\underline{\alpha_1}e_1 + \underline{\alpha_2}e_2 + \underline{\alpha_3}e_3 + \underline{\alpha_4}e_4 + \underline{\alpha_5}e_5) + \underline{\alpha_7}f_3 + \underline{\alpha_8}f_4 + \underline{\alpha_9}f_5, \\
    &\quad \quad -p(\sigma^2(\underline{\alpha_6\alpha_1})e_1 \!+\! \ldots \!+\!\sigma^2(\underline{\alpha_6\alpha_5})e_5 \!+\! \sigma^2(\underline{\alpha_7})f_3 \!+\! \sigma^2(\underline{\alpha_8})f_4 \!+\! \sigma^2(\underline{\alpha_9})f_5)\rangle \\
    &= -p(\underline{\alpha_6\alpha_3}\sigma^2(\underline{\alpha_7})\langle e_3, f_3 \rangle + \underline{\alpha_6\alpha_4}\sigma^2(\underline{\alpha_8})\langle e_4, f_4 \rangle + \underline{\alpha_6\alpha_5}\sigma^2(\underline{\alpha_9})\langle e_5, f_5 \rangle \\
    &\quad \quad + \underline{\alpha_7}\sigma^2(\underline{\alpha_6\alpha_3})\langle f_3, e_3\rangle + \underline{\alpha_8}\sigma^2(\underline{\alpha_6\alpha_4})\langle f_4, e_4\rangle + \underline{\alpha_9}\sigma^2(\underline{\alpha_6\alpha_5})\langle f_5, e_5\rangle) \\
    &= -p^{-1}(\underline{\alpha_6\alpha_3}\sigma^2(\underline{\alpha_7})+ \underline{\alpha_6\alpha_4}\sigma^2(\underline{\alpha_8})+\underline{\alpha_6\alpha_5}\sigma^2(\underline{\alpha_9}) \\
    &\quad \quad -\underline{\alpha_7}\sigma^2(\underline{\alpha_6\alpha_3}) - \underline{\alpha_8}\sigma^2(\underline{\alpha_6\alpha_4}) - \underline{\alpha_9}\sigma^2(\underline{\alpha_6\alpha_5})) \in W(k).
\end{align*}}

\noindent The second condition is automatically satisfied as the quasi-polarisation is alternating:
$$\langle w, FVw\rangle = \langle w,pw\rangle = p\langle w, w\rangle = 0 \in W(k).$$

\noindent Thus, to satisfy the non-trivial condition, the coefficient of $p^{-1}$ has to vanish. As we are working with Teichmüller lifts, this is equivalent to the equation
{\small\begin{align}\customanchor{eq:g5_2}{}
    \boxed{\alpha_6\!\left(\!-\alpha_3\alpha_7^{p^2}\!-\!\alpha_4\alpha_8^{p^2}\!-\!\alpha_5\alpha_9^{p^2} \!+\! \alpha_7\alpha_6^{p^2-1}\alpha_3^{p^2} \!+\! \alpha_8\alpha_6^{p^2-1}\alpha_4^{p^2}\!+\!\alpha_9\alpha_6^{p^2-1}\alpha_5^{p^2}\right) \!=\! 0,}
\end{align}}
which gives the fiber of the truncation morphism $\mathcal{V}_2\to\mathcal{V}_3$.

\begin{rem}
    By the same argument as in dimension $g=4$, choosing $\alpha_6 = 0$ gives rise to a non-rigid PFTQ and thus is a garbage component.
\end{rem}

We now turn to analysing $M_2$ by calculating its $a$-number. As we choose the parameters generically, we only compute it under this assumption:

\begin{pro}\label{pro:anumberM2}
Assume that the parameters are chosen generically, i.e.\ $[\alpha_1\mathpunct{:}\alpha_2\mathpunct{:}\alpha_3\mathpunct{:}\alpha_4\mathpunct{:}\alpha_5] \notin \mathbb{P}^4(\mathbb{F}_{p^2})$ and $[\alpha_6\mathpunct{:}\alpha_7\mathpunct{:}\alpha_8\mathpunct{:}\alpha_9] \notin \mathbb{P}^3(\mathbb{F}_{p^2})$. Then, it holds that $a(M_2) = 3$ if $\alpha_6 \neq 0$ (e.g.\ if the PFTQ is rigid).
\end{pro}

\begin{proof}
Recall that the $a$-number of $M_2$ is given by $\text{dim}_k\left(M_2/(F,V)M_2\right)$. In order to compute the dimension of this quotient space, we firstly need to find a basis of $M_2$. By the above construction, we have a generating set of $M_2$ that is given as follows:
$$M_2 = \langle \underbrace{\underline{\alpha_6}v + \underline{\alpha_7}f_3 + \underline{\alpha_8}f_4 + \underline{\alpha_9}f_5}_{=w}, Fv, Vv, pe_1, \ldots, pe_5, pf_1, \ldots, pf_5\rangle.$$

We may assume without loss of generality that $\alpha_1 = 1$ as the parameters $\alpha_1, \ldots, \alpha_5$ are completely symmetric. Moreover, $\alpha_6 \neq 0$ and thus we may further assume that $\alpha_6 = 1$ without loss of generality. We now show that we can discard three generators from the above generating set of $M_2$:
{\small\begin{align*}
    pe_1 \!&=\! pw - \underline{\alpha_2}pe_2 - \underline{\alpha_3}pe_3 - \underline{\alpha_4}pe_4 - \underline{\alpha_5}pe_5 - \underline{\alpha_7}pf_3 - \underline{\alpha_8}pf_4 - \underline{\alpha_9}pf_5, \\
    pf_1 \!&=\! pFv - \sigma(\underline{\alpha_2})pf_2-\sigma(\underline{\alpha_3})pf_3 - \sigma(\underline{\alpha_4})pf_4 - \sigma(\underline{\alpha_5})pf_5, \\
    pf_2 \!&=\! \tfrac{p}{\sigma(\underline{\alpha_2})-\sigma^{-1\!}(\underline{\alpha_2})}(Fv\!+\!Vv) \!-\! \tfrac{\sigma(\underline{\alpha_3})-\sigma^{-1\!}(\underline{\alpha_3})}{\sigma(\underline{\alpha_2})-\sigma^{-1\!}(\underline{\alpha_2})}pf_3 \!-\! \tfrac{\sigma(\underline{\alpha_4})-\sigma^{-1\!}(\underline{\alpha_4})}{\sigma(\underline{\alpha_2})-\sigma^{-1\!}(\underline{\alpha_2})}pf_4 \!-\! \tfrac{\sigma(\underline{\alpha_5})-\sigma^{-1\!}(\underline{\alpha_5})}{\sigma(\underline{\alpha_2})-\sigma^{-1\!}(\underline{\alpha_2})}pf_5,
\end{align*}}
where we used the genericity assumption $\alpha_2 \notin \mathbb{F}_{p^2}$. By these computations, the following is a basis of $M_2$ as a $W(k)$-module of rank $10$:
\begin{align*}
    M_2 = \langle w, Fv, Vv, pe_2, pe_3, pf_3, pe_4, pf_4, pe_5, pf_5\rangle.
\end{align*}

\noindent With this basis, the quotient algorithmically simplifies via the Smith normal form by using the generic parameter choices:
\begin{align*}
    \frac{M_2}{(F,V)M_2} \cong \langle [w], [pe_4], [pe_5]\rangle \cong k^3,
\end{align*}
which gives the claimed $a$-number and concludes the proof.
\end{proof}

\clearpage

\noindent We compute $M_1$ by descending one level down the PFTQ in two steps:
\begin{itemize}
    \item[(1)] Establish a candidate for $M_1$ by establishing a short exact sequence $$0 \to \frac{M_1}{(F,V)M_2} \to \frac{M_2}{(F,V)M_2} \to \frac{M_2}{M_1} \to 0.$$
    \item[(2)] Find the fiber of the truncation morphism via the descent criterion.
\end{itemize}

\vspace{4pt}

\noindent For (1), it is clear that the short exact sequence follows directly from the third isomorphism theorem by considering $$\frac{M_2/(F,V)M_2}{M_1/(F,V)M_2} \cong \frac{M_2}{M_1}.$$

\noindent Examining the dimensions of the terms in the sequence, we note that the middle term is of dimension $3$ by the previous proposition. Furthermore, the last term is of dimension $2$ by definition of the PFTQ. Thus, we obtain $$\text{dim}_k\left(\tfrac{M_1}{(F,V)M_2}\right) = \text{dim}_k\left(\tfrac{M_2}{(F,V)M_2}\right)- \text{dim}_k\left(\tfrac{M_2}{M_1}\right) = 3-2 = 1.$$

\noindent In the proof of Proposition \ref{pro:anumberM2}, we have seen that $\frac{M_2}{(F,V)M_2} \cong \langle [w], [pe_4], [pe_5]\rangle$. Therefore, the one-dimensional subspace is given by $$\frac{M_1}{(F,V)M_2} = \langle \alpha_{10}[w] + \alpha_{11}[pe_4] + \alpha_{12}[pe_5]\rangle,$$ 
where $[\alpha_{10}\mathpunct{:}\alpha_{11}\mathpunct{:}\alpha_{12}]\in \mathbb{P}^2.$ This can be lifted to a generating set of $M_1$:
$$M_1 = \langle \underbrace{\underline{\alpha_{10}}w + \underline{\alpha_{11}}pe_4 + \underline{\alpha_{12}}pe_5}_{=:u}, FM_2, VM_2\rangle.$$

\noindent For (2), we guarantee descent of the quasi-polarisation by applying the descent criterion at level $i=1$. Since this fixes $j=0$, we only verify that $$\langle u, Fu \rangle \in W(k).$$

\noindent Before examining this condition, we firstly precompute both $u$ and $Fu$:
{\small
\begin{align*}
    u &= \underline{\alpha_{10}\alpha_6\alpha_1}e_1 + \underline{\alpha_{10}\alpha_6\alpha_2}e_2 + \underline{\alpha_{10}\alpha_6\alpha_3}e_3 + (\underline{\alpha_{10}\alpha_6\alpha_4} + \underline{\alpha_{11}}p)e_4 \\
    &\quad + (\underline{\alpha_{10}\alpha_6\alpha_5} + \underline{\alpha_{12}}p)e_5 + \underline{\alpha_{10}\alpha_7}f_3 + \underline{\alpha_{10}\alpha_8}f_4 + \underline{\alpha_{10}\alpha_9}f_5,\\
    Fu &= \sigma(\underline{\alpha_{10}\alpha_6\alpha_1})f_1 + \sigma(\underline{\alpha_{10}\alpha_6\alpha_2})f_2 + \sigma(\underline{\alpha_{10}\alpha_6\alpha_3})f_3 + \sigma(\underline{\alpha_{10}\alpha_6\alpha_4}+\underline{\alpha_{11}}p)f_4 \\
    &\quad+ \sigma(\underline{\alpha_{10}\alpha_6\alpha_5}+\underline{\alpha_{12}}p)f_5 -\sigma(\underline{\alpha_{10}\alpha_7})pe_3 - \sigma(\underline{\alpha_{10}\alpha_8})pe_4 -\sigma(\underline{\alpha_{10}\alpha_9})pe_5.
\end{align*}}

\noindent Using the representation of the quasi-polarisation, we simplify the condition:
{\small
\begin{align*}
    \langle u, Fu \rangle \!&=\!\underline{\alpha_{10}\alpha_6\alpha_1}\sigma(\underline{\alpha_{10}\alpha_6\alpha_1})\langle e_1, f_1\rangle + \underline{\alpha_{10}\alpha_6\alpha_2}\sigma(\underline{\alpha_{10}\alpha_6\alpha_2})\langle e_2, f_2\rangle \\
    &\quad + \underline{\alpha_{10}\alpha_6\alpha_3}\sigma(\underline{\alpha_{10}\alpha_6\alpha_3})\langle e_3, f_3\rangle + (\underline{\alpha_{10}\alpha_6\alpha_4}\mathord{+}\underline{\alpha_{11}}p)\sigma(\underline{\alpha_{10}\alpha_6\alpha_4}\mathord{+}\underline{\alpha_{11}}p)\langle e_4, f_4\rangle \\
    &\quad + (\underline{\alpha_{10}\alpha_6\alpha_5}\mathord{+}\underline{\alpha_{12}}p)\sigma(\underline{\alpha_{10}\alpha_6\alpha_5}\mathord{+}\underline{\alpha_{12}}p)\langle e_5, f_5\rangle -\underline{\alpha_{10}\alpha_7}\sigma(\underline{\alpha_{10}\alpha_7})p\langle f_3, e_3\rangle \\
    &\quad - \underline{\alpha_{10}\alpha_8}\sigma(\underline{\alpha_{10}\alpha_8})p\langle f_4, e_4\rangle - \underline{\alpha_{10}\alpha_9}\sigma(\underline{\alpha_{10}\alpha_9})p\langle f_5, e_5\rangle \\
    &=\! \underline{\alpha_{11}}\sigma(\underline{\alpha_{11}}) \! +\!\underline{\alpha_{12}}\sigma(\underline{\alpha_{12}}) \!+\!
    p^{-1\!} \!\left[\underline{\alpha_{10}\alpha_7}\sigma(\underline{\alpha_{10}\alpha_7}) \!+\!\underline{\alpha_{10}\alpha_8}\sigma(\underline{\alpha_{10}\alpha_8}) \!+\! \underline{\alpha_{10}\alpha_9}\sigma(\underline{\alpha_{10}\alpha_9}) \right. \\
    &\quad \left. + \underline{\alpha_{10}\alpha_6\alpha_4}\sigma(\underline{\alpha_{11}}) +\underline{\alpha_{11}}\sigma(\underline{\alpha_{10}\alpha_6\alpha_4}) + \underline{\alpha_{10}\alpha_6\alpha_5}\sigma(\underline{\alpha_{12}}) + \underline{\alpha_{12}}\sigma(\underline{\alpha_{10}\alpha_6\alpha_5})\right] \\
    &\quad + p^{-2} \left[\underline{\alpha_{10}\alpha_6\alpha_1}\sigma(\underline{\alpha_{10}\alpha_6\alpha_1}) + \ldots + \underline{\alpha_{10}\alpha_6\alpha_5}\sigma(\underline{\alpha_{10}\alpha_6\alpha_5})\right] \in W(k).
\end{align*}}

\noindent To satisfy this condition, the coefficients of $p^{-2}$ and $p^{-1}$ have to vanish. As we are working with Teichmüller lifts, this is equivalent to the equations:
{\footnotesize\begin{equation}
\boxed{
\begin{gathered}
    \alpha_{10}^{p+1}\alpha_6^{p+1}\left(\alpha_1^{p+1}+\alpha_2^{p+1}+\alpha_3^{p+1}+\alpha_4^{p+1} + \alpha_5^{p+1}\right) = 0, \\
    \!\alpha_{10}\!\left(\alpha_6\alpha_4\alpha_{11}^p\!+\!\alpha_{11}\alpha_{10}^{p-1\!}\alpha_6^p\alpha_4^p \!+\! \alpha_6\alpha_5\alpha_{12}^p \!+\! \alpha_{12}\alpha_{10}^{p-1\!}\alpha_6^p\alpha_5^p
    \!+\! \alpha_{10}^p\alpha_7^{p+1\!} \!+\! \alpha_{10}^p\alpha_8^{p+1\!}\!+\!\alpha_{10}^p\alpha_9^{p+1\!}\right)\!=\!0,\!
\end{gathered}\customanchor{eq:g5_1}{}}
\end{equation}} 

\noindent giving the fiber of the truncation morphism $\mathcal{V}_1 \to \mathcal{V}_2$.

\vspace{8pt}

\noindent Having calculated this fiber, we make two remarks interpreting its structure: 

\begin{rem}
    Firstly, we observe that the first equation is automatically satisfied: Comparing with the equations imposed by the first step of the PFTQ, we see that the summand term vanishes. 

    The phenomenon that the first-step condition appears again can be explained intuitively: In this descent step, we are using the $(-2)^{\text{nd}}$-component of the Witt vectors for the first time. Therefore, it is not surprising to obtain a version of the condition imposed on the first step of a PFTQ from the corresponding coefficient.
\end{rem}

\begin{rem}
    The same argument as in dimension $g=4$ shows that choosing $\alpha_{10}=0$ leads to non-rigid PFTQs and thus is a garbage component. 
\end{rem}

\vspace{6pt}

With this, we have understood the construction of $M_1$. The last step of the PFTQ is always described by the black-box Theorem \ref{thm:laststepPFTQ}, which shows that the fiber of the truncation morphism $\mathcal{V}_0 \to \mathcal{V}_1$ is a $\mathbb{P}^1$-bundle. Consequently, the structure of $M_0$ is also determined, completing our analysis of $5$-dimensional PFTQs assuming generic parameter choices. 

\subsubsection{The non-supergeneral parameter choice}
We return to the point before assuming that the parameter is chosen generically and revisit this situation without imposing conditions on $[\alpha_1\mathpunct{:}\alpha_2\mathpunct{:}\alpha_3\mathpunct{:}\alpha_4\mathpunct{:}\alpha_5]$ this time. To compute $M_2$, we proceed in two steps, but now use an adapted sequence: 
\begin{itemize}
    \item[(1)] Find a candidate for $M_2$ by using the short exact sequence
    \begin{align*}
        0 \to \frac{M_2}{FM_3} \to \frac{M_3}{FM_3} \to \frac{M_3}{M_2} \to 0.
    \end{align*}
    \item[(2)] Apply a generalised version of the descent criterion to find the fiber of the truncation morphism.
\end{itemize}

\noindent For (1), the third isomorphism theorem gives the claimed exact sequence as $$\frac{M_3/FM_3}{M_2/FM_3} \cong \frac{M_3}{M_2}.$$

We compute the dimension of the first term to obtain information on $M_2$. By Lemma \ref{lem:PFTQLemma} (ii), the dimension of the middle term is given by the genus $g=5$ and is thus independent of the choice of parameter. As before, the dimension of the last term is fixed to be $3$ by definition of the PFTQ. Thus, 

$$\text{dim}_k\left(\tfrac{M_2}{FM_3}\right) = \text{dim}_k\left(\tfrac{M_3}{FM_3}\right)-\text{dim}_k\left(\tfrac{M_3}{M_2}\right) = 5-3 = 2.$$

\noindent It is an easy calculation to find five basis elements of $M_3/FM_3$, namely $$\frac{M_3}{FM_3}\cong \langle [\underline{\alpha_1}e_1 + \underline{\alpha_2}e_2 + \underline{\alpha_3}e_3 + \underline{\alpha_4}e_4 + \underline{\alpha_5}e_5], [f_2], [f_3], [f_4], [f_5]\rangle.$$ The two-dimensional subspace $M_2/FM_3$ of $M_3/FM_3$ is thus given by two non-zero elements, and without loss of generality they are of the form
\begin{align*}
    \frac{M_2}{FM_3} = \langle &\alpha_6[v] + \alpha_7[f_2] + \alpha_8[f_3] + \alpha_9[f_4] + \alpha_{10}[f_5], \\
    &\alpha_{11}[f_2] + \alpha_{12}[f_3] + \alpha_{13}[f_4] + \alpha_{14}[f_5]\rangle,
\end{align*}
where
$[\alpha_6\mathpunct{:}\alpha_7\mathpunct{:}\alpha_8\mathpunct{:}\alpha_9\mathpunct{:}\alpha_{10}]\in \mathbb{P}^4$ and $[\alpha_{11}\mathpunct{:}\alpha_{12}\mathpunct{:}\alpha_{13}\mathpunct{:}\alpha_{14}]\in\mathbb{P}^3$. We may lift this to a generating set of $M_2$ as follows: 
{\small
\begin{align*}
    M_2 = \langle &\underbrace{\underline{\alpha_6}v + \underline{\alpha_7}f_2 + \underline{\alpha_8}f_3 + \underline{\alpha_9}f_4 + \underline{\alpha_{10}}f_5}_{=:w_1}, \underbrace{\underline{\alpha_{11}}f_2 + \underline{\alpha_{12}}f_3 + \underline{\alpha_{13}}f_4 + \underline{\alpha_{14}}f_5}_{=:w_2},  FM_3\rangle.
\end{align*}
}

For (2), we want to ensure that the polarisation descends to $M_2$ and hence apply the generalised descent criterion at level $i=2$ fixing $j = 0, 1$. As we add two Witt vectors, we have to check the following eight conditions:
\begin{align*}
    &\underline{j=0}: \langle w_1, F^2w_1\rangle, \langle w_1, F^2w_2\rangle, \langle w_2, F^2w_1\rangle, \langle w_2, F^2w_2\rangle \in W(k) \\
    &\underline{j=1}: \langle w_1, FVw_1\rangle, \langle w_1, FVw_2\rangle, \langle w_2, FVw_1\rangle, \langle w_2, FVw_2\rangle \in W(k)
\end{align*}

\noindent Before examining the conditions, we compute $w_1, w_2$ and their $F^2$-images:
{\small
\begin{align*}
    w_1 &= \underline{\alpha_6\alpha_1}e_1 + \underline{\alpha_6\alpha_2}e_2 + \underline{\alpha_6\alpha_3}e_3 + \underline{\alpha_6\alpha_4}e_4 + \underline{\alpha_6\alpha_5}e_5 + \underline{\alpha_7}f_2 + \underline{\alpha_8}f_3 + \underline{\alpha_9}f_4 + \underline{\alpha_{10}}f_5, \\
    w_2 &= \underline{\alpha_{11}}f_2 + \underline{\alpha_{12}}f_3 + \underline{\alpha_{13}}f_4 + \underline{\alpha_{14}}f_5, \\
    F^2w_1 &= -p\sigma^2(\underline{\alpha_6\alpha_1})e_1 -p\sigma^2(\underline{\alpha_6\alpha_2})e_2 -p\sigma^2(\underline{\alpha_6\alpha_3})e_3 -p\sigma^2(\underline{\alpha_6\alpha_4})e_4 -p\sigma^2(\underline{\alpha_6\alpha_5})e_5 \\
    &\phantom{=} -p\sigma^2(\underline{\alpha_7})f_2 - p\sigma^2(\underline{\alpha_8})f_3 - p\sigma^2(\underline{\alpha_9})f_4 - p\sigma^2(\underline{\alpha_{10}})f_5\\
    F^2w_2 &= - p\sigma^2(\underline{\alpha_{11}})f_2 - p\sigma^2(\underline{\alpha_{12}})f_3 - p\sigma^2(\underline{\alpha_{13}})f_4 - p\sigma^2(\underline{\alpha_{14}})f_5
\end{align*}
}

We start with $j=0$, where the first condition simplifies as follows:
{\small
\begin{align*}
\langle w_1, F^2w_1\rangle &= -\underline{\alpha_6\alpha_2}\sigma^2(\underline{\alpha_7})p\langle e_2, f_2\rangle - \ldots -  \underline{\alpha_6\alpha_5}\sigma^2(\underline{\alpha_{10}})p\langle e_5, f_5\rangle \\
& \phantom{=} - \underline{\alpha_7}\sigma^2(\underline{\alpha_6\alpha_2})p\langle f_2, e_2\rangle - \ldots - \underline{\alpha_{10}}\sigma^2(\underline{\alpha_6\alpha_5})p\langle f_5, e_5\rangle \\
&= p^{-1} \left(-\underline{\alpha_6\alpha_2}\sigma^2(\underline{\alpha_7})-\ldots - \underline{\alpha_6\alpha_5}\sigma^2(\underline{\alpha_{10}}) \right.\\
&\phantom{=}\quad \quad \quad \left.+ \underline{\alpha_7}\sigma^2(\underline{\alpha_6\alpha_2}) + \ldots + \underline{\alpha_{10}}\sigma^2(\underline{\alpha_6\alpha_5})\right) \in W(k).
\end{align*}
}

\noindent Computing the second condition, we obtain
{\small
\begin{align*}
\langle w_1, F^2w_2\rangle &= - \underline{\alpha_6\alpha_2}\sigma^2(\underline{\alpha_{11}})p\langle e_2, f_2\rangle - \ldots -\underline{\alpha_6\alpha_5}\sigma^2(\underline{\alpha_{14}})p\langle e_5, f_5\rangle \\
&= -p^{-1}\left(\underline{\alpha_6\alpha_2}\sigma^2(\underline{\alpha_{11}})+ \ldots + \underline{\alpha_6\alpha_5}\sigma^2(\underline{\alpha_{14}})\right) \in W(k).
\end{align*}
}
Similarly, the third condition simplifies to
{\small
\begin{align*}
\langle w_2, F^2w_1\rangle &= -\underline{\alpha_{11}}\sigma^2(\underline{\alpha_6\alpha_2})p\langle f_2, e_2\rangle - \ldots - \underline{\alpha_{14}}\sigma^2(\underline{\alpha_6\alpha_5})p\langle f_5, e_5\rangle \\
&= p^{-1}\left(\underline{\alpha_{11}}\sigma^2(\underline{\alpha_6\alpha_2}) + \ldots + \alpha_{14}\sigma^2(\underline{\alpha_6\alpha_5})\right) \in W(k).
\end{align*}
}

\noindent Finally, the last condition for $j=0$ is automatically satisfied:
{\small
\begin{align*}
    \langle w_2, F^2w_2\rangle &= \langle \underline{\alpha_{11}}f_2 + \underline{\alpha_{12}}f_3 + \underline{\alpha_{13}}f_4 + \underline{\alpha_{14}}f_5, \\
    &\quad - p\sigma^2(\underline{\alpha_{11}})f_2 - p\sigma^2(\underline{\alpha_{12}})f_3 - p\sigma^2(\underline{\alpha_{13}})f_4 - p\sigma^2(\underline{\alpha_{14}})f_5\rangle = 0 \in W(k)
\end{align*}
}

\noindent as the quasi-polarisation gives $\langle f_i, f_j \rangle = 0$ for any $i,j$. 

Continuing with $j=1$, the first condition is also automatically satisfied:
{\small
\begin{align*}
    \langle w_1, FVw_1\rangle = \langle w_1, pw_1\rangle = p \langle w_1, w_1\rangle =0 \in W(k)
\end{align*}
}

\noindent as the quasi-polarisation is alternating. 

\noindent The second condition gives a non-vanishing contribution, namely
{\small
\begin{align*}
    \langle w_1, FVw_2\rangle &= p \langle w_1, w_2 \rangle = p \underline{\alpha_6}\langle v, \underline{\alpha_{11}}f_2 + \underline{\alpha_{12}}f_3 + \underline{\alpha_{13}}f_4 + \underline{\alpha_{14}}f_5\rangle \\
    &= p^{-1}\left(\underline{\alpha_6\alpha_2\alpha_{11}} +\underline{\alpha_6\alpha_3\alpha_{12}} + \underline{\alpha_6\alpha_4\alpha_{13}} + \underline{\alpha_6\alpha_5\alpha_{14}}\right) \in W(k). 
\end{align*}
}

\noindent The third condition gives the same condition as the second one since $$\langle w_2, FVw_1\rangle = p\langle w_2, w_1\rangle = -p\langle w_1,w_2\rangle = -\langle w_1, FVw_2\rangle.$$

\noindent Lastly, the fourth condition is also automatically satisfied as a result of the alternating property of the quasi-polarisation: 
{\small
\begin{align*}
    \langle w_2, FVw_2 \rangle = \langle w_2, pw_2\rangle = p \langle w_2, w_2 \rangle = 0 \in W(k).
\end{align*}
}

Collecting all of the conditions, we arrive at the following set of equations: 
\makebox[\textwidth][c]{
\begin{minipage}{1.2\textwidth}
{\footnotesize\begin{equation}
\boxed{
\begin{gathered}
\alpha_6(-\alpha_2\alpha_7^{p^2}-\alpha_3\alpha_8^{p^2}-\alpha_4\alpha_9^{p^2}-\alpha_5\alpha_{10}^{p^2} +\alpha_7\alpha_6^{p^2-1}\alpha_2^{p^2}+\alpha_8\alpha_6^{p^2-1}\alpha_3^{p^2}+\alpha_9\alpha_6^{p^2-1}\alpha_4^{p^2}+\alpha_{10}\alpha_6^{p^2-1}\alpha_5^{p^2})= 0, \\
\alpha_6(-\alpha_2\alpha_{11}^{p^2}-\alpha_3\alpha_{12}^{p^2}-\alpha_4\alpha_{13}^{p^2}-\alpha_5\alpha_{14}^{p^2})= 0, \\
\alpha_6^{p^2}(\alpha_{11}\alpha_2^{p^2}+\alpha_{12}\alpha_3^{p^2}+\alpha_{13}\alpha_4^{p^2}+\alpha_{14}\alpha_5^{p^2}) = 0, \\
\alpha_6(\alpha_2\alpha_{11} + \alpha_3\alpha_{12} + \alpha_4\alpha_{13} + \alpha_5\alpha_{14}) = 0,
\end{gathered}}
\end{equation}}
\end{minipage}
}
\\

\noindent giving the fiber of the truncation morphism $\mathcal{V}_2 \to \mathcal{V}_3$.

\begin{rem}
    By the same argument as before, choosing $\alpha_6 = 0$ gives rise to a non-rigid PFTQ and thus is a garbage component.
\end{rem}

\noindent Now, we compute $M_1$ by descending one level down the PFTQ in two steps:
\begin{itemize}
    \item[(1)] Establish a candidate for $M_1$ by establishing a short exact sequence $$0 \to \frac{M_1}{FM_2} \to \frac{M_2}{FM_2} \to \frac{M_2}{M_1} \to 0.$$
    \item[(2)] Find the fiber of the truncation morphism via the descent criterion.
\end{itemize}

\noindent For (1), the third isomorphism theorem gives the short exact sequence as $$\frac{M_2/FM_2}{M_1/FM_2} \cong \frac{M_2}{M_1}.$$

We compute the dimension of the first term to gain information on $M_1$. 
As before, Lemma \ref{lem:PFTQLemma} (ii) gives the dimension of the middle term whilst the dimension of the last term is fixed by definition of the PFTQ. Therefore,

$$\text{dim}_k\left(\tfrac{M_1}{FM_2}\right) = \text{dim}_k\left(\tfrac{M_2}{FM_2}\right)-\text{dim}_k\left(\tfrac{M_2}{M_1}\right) = 5-2 = 3.$$

We now determine a basis for $M_2/FM_2$. Since the PFTQ is rigid, we may assume without loss of generality that $\alpha_6 = 1$. Moreover, since the parameters $[\alpha_{11}\mathpunct{:}\alpha_{12}\mathpunct{:}\alpha_{13}\mathpunct{:}\alpha_{14}]$ are completely symmetric, we may also assume that $\alpha_{11}=1$. Under these assumptions, a straightforward computation shows that a basis for $M_2/FM_2$ is given by
\begin{align*}
    \frac{M_2}{FM_2} \cong \langle [w_1],[w_2], [pe_3], [pe_4], [pe_5]\rangle.
\end{align*}

The three-dimensional subspace $M_1/FM_2$ of $M_2/FM_2$ is thus given by three non-zero elements, and without loss of generality they are of the form
{\small\begin{align*}
    \frac{M_1}{FM_2} = \langle &\alpha_{15}[w_1] + \alpha_{16}[w_2] + \alpha_{17}[pe_3] + \alpha_{18}[pe_4] + \alpha_{19}[pe_5], \\
    &\alpha_{20}[w_2] + \alpha_{21}[pe_3] + \alpha_{22}[pe_4] + \alpha_{23}[pe_5], 
    \alpha_{24}[pe_3] + \alpha_{25}[pe_4] + \alpha_{26}[pe_5]\rangle,
\end{align*}}

\noindent with parameters $[\alpha_{15}\mathpunct{:}\alpha_{16}\mathpunct{:}\alpha_{17}\mathpunct{:}\alpha_{18}\mathpunct{:}\alpha_{19}]\in \mathbb{P}^4$, $[\alpha_{20}\mathpunct{:}\alpha_{21}\mathpunct{:}\alpha_{22}\mathpunct{:}\alpha_{23}]\in\mathbb{P}^3$, and $[\alpha_{24}\mathpunct{:}\alpha_{25}\mathpunct{:}\alpha_{26}]\in \mathbb{P}^2$. We may lift this to a generating set of $M_1$ as follows: 
{\small
\begin{align*}
    M_1 = \langle &\underbrace{\underline{\alpha_{15}}w_1 + \underline{\alpha_{16}}w_2 + \underline{\alpha_{17}}pe_3 + \underline{\alpha_{18}}pe_4 + \underline{\alpha_{19}}pe_5}_{=:u_1},\\ &\underbrace{\underline{\alpha_{20}}w_2 + \underline{\alpha_{21}}pe_3 + \underline{\alpha_{22}}pe_4 + \underline{\alpha_{23}}pe_5}_{=:u_2},
    \underbrace{\underline{\alpha_{24}}pe_3 + \underline{\alpha_{25}}pe_4 + \underline{\alpha_{26}}pe_5}_{=:u_3}, FM_2\rangle.
\end{align*}
}

For (2), we want to ensure that the polarisation descends to $M_1$ and hence apply the generalised descent criterion at level $i=1$ fixing $j = 0$. As we add three Witt vectors, we have to check the following nine conditions:
\begin{align*}
    &\langle u_1, Fu_1\rangle, \langle u_1, Fu_2\rangle, \langle u_1, Fu_3\rangle \in W(k) \\
    &\langle u_2, Fu_1\rangle, \langle u_2, Fu_2\rangle, \langle u_2, Fu_3\rangle \in W(k) \\
    &\langle u_3, Fu_1\rangle, \langle u_3, Fu_2\rangle, \langle u_3, Fu_3\rangle \in W(k)
\end{align*}

\noindent Before examining the conditions, we compute $u_1, u_2, u_3$.
{\small
\begin{align*}
    u_1 &= \underline{\alpha_1\alpha_6\alpha_{15}}e_1 + \underline{\alpha_2\alpha_6\alpha_{15}}e_2 + (\underline{\alpha_3\alpha_6\alpha_{15}} + \underline{\alpha_{17}p)}e_3 + (\underline{\alpha_4\alpha_6\alpha_{15}} + \underline{\alpha_{18}p)}e_4 \\
    &\phantom{=}+ (\underline{\alpha_5\alpha_6\alpha_{15}} + \underline{\alpha_{19}p)}e_5 
    + (\underline{\alpha_7\alpha_{15}}+\underline{\alpha_{11}\alpha_{16}})f_2 + (\underline{\alpha_8\alpha_{15}}+\underline{\alpha_{12}\alpha_{16}})f_3 \\
    &\phantom{=}+ (\underline{\alpha_9\alpha_{15}}+\underline{\alpha_{13}\alpha_{16}})f_4 + (\underline{\alpha_{10}\alpha_{15}}+\underline{\alpha_{14}\alpha_{16}})f_5, \\
    u_2 &= \underline{\alpha_{21}}pe_3 + \underline{\alpha_{22}}pe_4 + \underline{\alpha_{23}}pe_5 + \underline{\alpha_{11}\alpha_{20}}f_2 + \underline{\alpha_{12}\alpha_{20}}f_3 + \underline{\alpha_{13}\alpha_{20}}f_4 + \underline{\alpha_{14}\alpha_{20}}f_5, \\
    u_3 &= \underline{\alpha_{24}}pe_3 + \underline{\alpha_{25}}pe_4 + \underline{\alpha_{26}}pe_5.
\end{align*}
}
Similarly, we calculate their images under the Frobenius map:
{\small
\begin{align*}
    Fu_1 &= -(\sigma(\underline{\alpha_7\alpha_{15}})+\sigma(\underline{\alpha_{11}\alpha_{16}}))pe_2 - (\sigma(\underline{\alpha_8\alpha_{15}})+\sigma(\underline{\alpha_{12}\alpha_{16}}))pe_3 \\
    &\phantom{=} - (\sigma(\underline{\alpha_9\alpha_{15}})+\sigma(\underline{\alpha_{13}\alpha_{16}}))pe_4 - (\sigma(\underline{\alpha_{10}\alpha_{15}})+\sigma(\underline{\alpha_{14}\alpha_{16}}))pe_5 \\
    &\phantom{=} +\sigma(\underline{\alpha_1\alpha_6\alpha_{15}})f_1 +\sigma(\underline{\alpha_2\alpha_6\alpha_{15}})f_2 + (\sigma(\underline{\alpha_3\alpha_6\alpha_{15}})+\sigma(\underline{\alpha_{17}}p))f_3 \\
    &\phantom{=} + (\sigma(\underline{\alpha_4\alpha_6\alpha_{15}})+\sigma(\underline{\alpha_{18}}p))f_4 + (\sigma(\underline{\alpha_5\alpha_6\alpha_{15}})+\sigma(\underline{\alpha_{19}}p))f_5 \\
    Fu_2 &= -\sigma(\underline{\alpha_{11}\alpha_{20}})pe_2 -\sigma(\underline{\alpha_{12}\alpha_{20}})pe_3 -\sigma(\underline{\alpha_{13}\alpha_{20}})pe_4 -\sigma(\underline{\alpha_{14}\alpha_{20}})pe_5 \\
    &\phantom{=} + \sigma(\underline{\alpha_{21}})pf_3 + \sigma(\underline{\alpha_{22}})pf_4 + \sigma(\underline{\alpha_{23}})pf_5 \\
    Fu_3 &= \sigma(\underline{\alpha_{24}})pf_3 + \sigma(\underline{\alpha_{25}})pf_4 + \sigma(\underline{\alpha_{26}})pf_5
\end{align*}
}

\noindent Using the representation of the quasi-polarisation, the first condition gives
{\small
\begin{align*}
    \langle u_1, Fu_1\rangle 
    &= p^{-2} \left[\underline{\alpha_1\alpha_6\alpha_{15}}\sigma(\underline{\alpha_1\alpha_6\alpha_{15}}) + \underline{\alpha_2\alpha_6\alpha_{15}}\sigma(\underline{\alpha_2\alpha_6\alpha_{15}}) + \underline{\alpha_3\alpha_6\alpha_{15}}\sigma(\underline{\alpha_3\alpha_6\alpha_{15}}) \right.\\
    & \left. \quad \quad  + \underline{\alpha_4\alpha_6\alpha_{15}}\sigma(\underline{\alpha_4\alpha_6\alpha_{15}}) +
    \underline{\alpha_5\alpha_6\alpha_{15}}\sigma(\underline{\alpha_5\alpha_6\alpha_{15}})
    \right] \\
    &+ p^{-1} \left[\underline{\alpha_3\alpha_6\alpha_{15}}\sigma(\underline{\alpha_{17}}) + \underline{\alpha_{17}}\sigma(\underline{\alpha_3\alpha_6\alpha_{15}}) + \underline{\alpha_4\alpha_6\alpha_{15}}\sigma(\underline{\alpha_{18}}) \right.\\
    & \left. \quad \quad + \underline{\alpha_{18}}\sigma(\underline{\alpha_4\alpha_6\alpha_{15}}) + \underline{\alpha_5\alpha_6\alpha_{15}}\sigma(\underline{\alpha_{19}}) + \underline{\alpha_{19}}\sigma(\underline{\alpha_5\alpha_6\alpha_{15}}) \right.\\
    & \left. \quad \quad +\underline{\alpha_7\alpha_{15}}\sigma(\underline{\alpha_7\alpha_{15}})+\underline{\alpha_7\alpha_{15}}\sigma(\underline{\alpha_{11}\alpha_{16}}) + \underline{\alpha_{11}\alpha_{16}}\sigma(\underline{\alpha_7\alpha_{15}}) + \underline{\alpha_{11}\alpha_{16}}\sigma(\underline{\alpha_{11}\alpha_{16}}) \right. \\
    & \left. \quad \quad +\underline{\alpha_8\alpha_{15}}\sigma(\underline{\alpha_8\alpha_{15}})+\underline{\alpha_8\alpha_{15}}\sigma(\underline{\alpha_{12}\alpha_{16}}) + \underline{\alpha_{12}\alpha_{16}}\sigma(\underline{\alpha_8\alpha_{15}}) + \underline{\alpha_{12}\alpha_{16}}\sigma(\underline{\alpha_{12}\alpha_{16}}) \right. \\
    & \left. \quad \quad  +\underline{\alpha_9\alpha_{15}}\sigma(\underline{\alpha_9\alpha_{15}})+\underline{\alpha_9\alpha_{15}}\sigma(\underline{\alpha_{13}\alpha_{16}}) + \underline{\alpha_{13}\alpha_{16}}\sigma(\underline{\alpha_9\alpha_{15}}) + \underline{\alpha_{13}\alpha_{16}}\sigma(\underline{\alpha_{13}\alpha_{16}}) \right. \\
    & \left. \quad \quad +\underline{\alpha_{10}\alpha_{15}}\sigma(\underline{\alpha_{10}\alpha_{15}})+\underline{\alpha_{10}\alpha_{15}}\sigma(\underline{\alpha_{14}\alpha_{16}}) + \underline{\alpha_{14}\alpha_{16}}\sigma(\underline{\alpha_{10}\alpha_{15}}) + \underline{\alpha_{14}\alpha_{16}}\sigma(\underline{\alpha_{14}\alpha_{16}})
    \right]\\
    &+\left[\underline{\alpha_{17}}\sigma(\underline{\alpha_{17}})+ \underline{\alpha_{18}}\sigma(\underline{\alpha_{18}}) + \underline{\alpha_{19}}\sigma(\underline{\alpha_{19}})\right] \in W(k).
\end{align*}
}

\noindent Computing the second condition, we obtain
{\small
\begin{align*}
    \langle u_1, Fu_2\rangle &= p^{-1} \left[\underline{\alpha_3\alpha_6\alpha_{15}}\sigma(\underline{\alpha_{21}}) + \underline{\alpha_4\alpha_6\alpha_{15}}\sigma(\underline{\alpha_{22}}) + \underline{\alpha_5\alpha_6\alpha_{15}}\sigma(\underline{\alpha_{23}}) \right.\\
    &\left. \quad \quad \quad + \underline{\alpha_7\alpha_{15}}\sigma(\underline{\alpha_{11}\alpha_{20}})+\underline{\alpha_{11}\alpha_{16}}\sigma(\underline{\alpha_{11}\alpha_{20}}) + \underline{\alpha_8\alpha_{15}}\sigma(\underline{\alpha_{12}\alpha_{20}}) \right. \\
    &\left. \quad \quad \quad +\underline{\alpha_{12}\alpha_{16}}\sigma(\underline{\alpha_{12}\alpha_{20}}) + \underline{\alpha_9\alpha_{15}}\sigma(\underline{\alpha_{13}\alpha_{20}})+ \underline{\alpha_{13}\alpha_{16}}\sigma(\underline{\alpha_{13}\alpha_{20}}) \right. \\
    &\left. \quad \quad \quad + \underline{\alpha_{10}\alpha_{15}}\sigma(\underline{\alpha_{14}\alpha_{20}}) + \underline{\alpha_{14}\alpha_{16}}\sigma(\underline{\alpha_{14}\alpha_{20}}) \right] \\
    &\phantom{=}+\left[\underline{\alpha_{17}}\sigma(\underline{\alpha_{21}}) + \underline{\alpha_{18}}\sigma(\underline{\alpha_{22}}) + \underline{\alpha_{19}}\sigma(\underline{\alpha_{23}}) \right] \in W(k).
\end{align*} 
}
The third condition simplifies as follows:
{\small
\begin{align*}
    \langle u_1, Fu_3\rangle &= p^{-1}\left[\underline{\alpha_3\alpha_6\alpha_{15}}\sigma(\underline{\alpha_{24}}) + \underline{\alpha_4\alpha_6\alpha_{15}} \sigma(\underline{\alpha_{25}}) + \underline{\alpha_5\alpha_6\alpha_{15}} \sigma(\underline{\alpha_{26}})\right] \\
    &\phantom{=}+\underline{\alpha_{17}}\sigma(\underline{\alpha_{24}}) +\underline{\alpha_{18}}\sigma(\underline{\alpha_{25}}) +\underline{\alpha_{19}}\sigma(\underline{\alpha_{26}}) \in W(k).
\end{align*}
}
\noindent Similarly, we compute the fourth condition:
{\small
\begin{align*}
    \langle u_2, Fu_1\rangle &= p^{-1} \left[\sigma(\underline{\alpha_3\alpha_6\alpha_{15}})\underline{\alpha_{21}} + \sigma(\underline{\alpha_4\alpha_6\alpha_{15}})\underline{\alpha_{22}} + \sigma(\underline{\alpha_5\alpha_6\alpha_{15}})\underline{\alpha_{23}} \right.\\
    &\left. \quad \quad \quad + \sigma(\underline{\alpha_7\alpha_{15}})\underline{\alpha_{11}\alpha_{20}}+\sigma(\underline{\alpha_{11}\alpha_{16}})\underline{\alpha_{11}\alpha_{20}} + \sigma(\underline{\alpha_8\alpha_{15}})\underline{\alpha_{12}\alpha_{20}} \right. \\
    &\left. \quad \quad \quad +\sigma(\underline{\alpha_{12}\alpha_{16}})\underline{\alpha_{12}\alpha_{20}} + \sigma(\underline{\alpha_9\alpha_{15}})\underline{\alpha_{13}\alpha_{20}}+ \sigma(\underline{\alpha_{13}\alpha_{16}})\underline{\alpha_{13}\alpha_{20}} \right. \\
    &\left. \quad \quad \quad + \sigma(\underline{\alpha_{10}\alpha_{15}})\underline{\alpha_{14}\alpha_{20}} + \sigma(\underline{\alpha_{14}\alpha_{16}})\underline{\alpha_{14}\alpha_{20}} \right] \\
    &\phantom{=}+\left[\sigma(\underline{\alpha_{17}})\underline{\alpha_{21}} + \sigma(\underline{\alpha_{18}})\underline{\alpha_{22}} + \sigma(\underline{\alpha_{19}})\underline{\alpha_{23}} \right] \in W(k).
\end{align*} 
}
The fifth condition yields the following expression:
{\small
\begin{align*}
    \langle u_2, Fu_2\rangle = p^{-1}\left[\underline{\alpha_{11}\alpha_{20}}\sigma(\underline{\alpha_{11}\alpha_{20}}) + \ldots + \underline{\alpha_{14}\alpha_{20}}\sigma(\underline{\alpha_{14}\alpha_{20}})\right] \in W(k).
\end{align*}
}
The sixth condition is automatically satisfied as
{\small
\begin{align*}
    \langle u_2, Fu_3\rangle = \underline{\alpha_{21}}\sigma(\underline{\alpha_{24}}) + \underline{\alpha_{22}}\sigma(\underline{\alpha_{25}})+ \underline{\alpha_{23}}\sigma(\underline{\alpha_{26}})\in W(k).
\end{align*}
}
The last non-trivial condition is the seventh one:
{\small
\begin{align*}
    \langle u_3, Fu_1\rangle &= p^{-1}[\underline{\alpha_{24}} \sigma(\underline{\alpha_3\alpha_6\alpha_{15}}) + \underline{\alpha_{25}} \sigma(\underline{\alpha_4\alpha_6\alpha_{15}}) + \underline{\alpha_{26}} \sigma(\underline{\alpha_5\alpha_6\alpha_{15}})] \\
    &\phantom{=} +\underline{\alpha_{24}}\sigma(\underline{\alpha_{17}}) + \underline{\alpha_{25}}\sigma(\underline{\alpha_{18}}) + \underline{\alpha_{26}}\sigma(\underline{\alpha_{19}})\in W(k).
\end{align*}
}
The eighth and ninth conditions are automatically satisfied as
{\small
\begin{align*}
    \langle u_3, Fu_2\rangle &= \underline{\alpha_{24}}\sigma(\underline{\alpha_{21}}) + \underline{\alpha_{25}}\sigma(\underline{\alpha_{22}})+ \underline{\alpha_{26}}\sigma(\underline{\alpha_{23}})\in W(k). \\
    \langle u_3, Fu_3\rangle &= \underline{\alpha_{24}}\sigma(\underline{\alpha_{24}}) + \underline{\alpha_{25}}\sigma(\underline{\alpha_{25}})+ \underline{\alpha_{26}}\sigma(\underline{\alpha_{26}})\in W(k).
\end{align*}
}

To satisfy the conditions, the coefficients of $p^{-2}$ and $p^{-1}$ have to vanish. As we are working with Teichmüller lifts, this is equivalent to the equations:

\makebox[\textwidth][c]{
\begin{minipage}{1.4\textwidth}
{\footnotesize\begin{equation}\label{eq:arb5}
\boxed{
\begin{gathered}
\alpha_6^{p+1}\alpha_{15}^{p+1}(\alpha_1^{p+1} + \alpha_2^{p+1} + \alpha_3^{p+1} + \alpha_4^{p+1} + \alpha_5^{p+1})= 0, \\
\left\{\begin{gathered}
\alpha_3\alpha_6\alpha_{15}\alpha_{17}^p+\alpha_{17}\alpha_3^p\alpha_6^p\alpha_{15}^{p} + \alpha_4\alpha_6\alpha_{15}\alpha_{18}^p + \alpha_{18}\alpha_4^p\alpha_6^p\alpha_{15}^{p} + \alpha_5\alpha_6\alpha_{15}\alpha_{19}^p + \alpha_{19}\alpha_5^p\alpha_6^p\alpha_{15}^{p} \\
+\alpha_7^{p+1}\alpha_{15}^{p+1} + \alpha_7\alpha_{15}\alpha_{11}^p\alpha_{16}^p + \alpha_{11}\alpha_{16}\alpha_7^p\alpha_{15}^{p} + \alpha_{11}^{p+1}\alpha_{16}^{p+1} + \alpha_8^{p+1}\alpha_{15}^{p+1} + \alpha_8\alpha_{15}\alpha_{12}^p\alpha_{16}^p + \alpha_{12}\alpha_{16}\alpha_8^p\alpha_{15}^{p} + \alpha_{12}^{p+1}\alpha_{16}^{p+1}\\
+\alpha_9^{p+1}\alpha_{15}^{p+1} + \alpha_9\alpha_{15}\alpha_{13}^p\alpha_{16}^p + \alpha_{13}\alpha_{16}\alpha_9^p\alpha_{15}^{p} + \alpha_{13}^{p+1}\alpha_{16}^{p+1} + 
\alpha_{10}^{p+1}\alpha_{15}^{p+1} + \alpha_{10}\alpha_{15}\alpha_{14}^p\alpha_{16}^p + \alpha_{14}\alpha_{16}\alpha_{10}^p\alpha_{15}^{p} + \alpha_{14}^{p+1}\alpha_{16}^{p+1} 
\end{gathered} \right\}
= 0,\\
\left\{\begin{gathered}\alpha_3\alpha_6\alpha_{15}\alpha_{21}^p + \alpha_4\alpha_6\alpha_{15}\alpha_{22}^p + \alpha_5\alpha_6\alpha_{15}\alpha_{23}^p + \alpha_7\alpha_{15}\alpha_{11}^p\alpha_{20}^p + \alpha_{16}\alpha_{11}^{p+1}\alpha_{20}^p + \alpha_8\alpha_{15}\alpha_{12}^p\alpha_{20}^p \\
+\alpha_{12}^{p+1}\alpha_{16}\alpha_{20}^p + \alpha_9\alpha_{15}\alpha_{13}^p\alpha_{20}^p + \alpha_{13}^{p+1}\alpha_{16}\alpha_{20}^p + \alpha_{10}\alpha_{15}\alpha_{14}^p\alpha_{20}^p + \alpha_{14}^{p+1}\alpha_{16}\alpha_{20}^p 
\end{gathered} \right\}
= 0,\\
\alpha_6\alpha_{15}(\alpha_3\alpha_{24}^p + \alpha_4\alpha_{25}^p + \alpha_5\alpha_{26}^p) = 0,\\
\left\{\begin{gathered}
\alpha_3^p\alpha_6^p\alpha_{15}^p\alpha_{21} + \alpha_4^p\alpha_6^p\alpha_{15}^p\alpha_{22} + \alpha_5^p\alpha_6^p\alpha_{15}^p\alpha_{23} + \alpha_7^p\alpha_{15}^p\alpha_{11}\alpha_{20} + \alpha_{16}^p\alpha_{11}^{p+1}\alpha_{20} + \alpha_8^p\alpha_{15}^p\alpha_{12}\alpha_{20} \\
+\alpha_{12}^{p+1}\alpha_{16}^p\alpha_{20} + \alpha_9^p\alpha_{15}^p\alpha_{13}\alpha_{20} + \alpha_{13}^{p+1}\alpha_{16}^p\alpha_{20} + \alpha_{10}^p\alpha_{15}^p\alpha_{14}\alpha_{20} + \alpha_{14}^{p+1}\alpha_{16}^p\alpha_{20}
\end{gathered} \right\}
= 0,\\
\alpha_{20}^{p+1}(\alpha_{11}^{p+1} + \alpha_{12}^{p+1} + \alpha_{13}^{p+1} + \alpha_{14}^{p+1}) = 0,\\
\alpha_6^p\alpha_{15}^p(\alpha_3^p\alpha_{24} + \alpha_4^p\alpha_{25} + \alpha_5^p\alpha_{26}) = 0.
\end{gathered}}
\end{equation}}
\end{minipage}
}

\begin{rem}
    By the same argument as before, we see that simultaneously choosing $\alpha_{15} = \alpha_{16} = \alpha_{20} = 0$ gives rise to a non-rigid PFTQ and thus defines a garbage component.  
\end{rem}

\vspace{6pt}

With this, we have understood the construction of $M_1$. The last step of the PFTQ is always described by the black-box Theorem \ref{thm:laststepPFTQ}, which shows that the fiber of the truncation morphism $\mathcal{V}_0 \to \mathcal{V}_1$ is a $\mathbb{P}^1$-bundle. Consequently, the structure of $M_0$ is also determined, completing our analysis of $5$-dimensional PFTQs for arbitrary parameter choices.

\printbibliography{}

\end{document}